\documentclass[11pt,a4paper]{article}

\usepackage{amsmath,amssymb,amsthm,amsbsy,amsfonts}
\usepackage{pdfrender,xcolor}
\usepackage{tabu}
\usepackage{verbatim}
\usepackage{float}
\usepackage{dsfont}

\PassOptionsToPackage{linktocpage}{hyperref}

\usepackage[colorlinks=true, linkcolor=blue, urlcolor=blue,citecolor=red,pagebackref=true]{hyperref}

\usepackage{tikz}
\usepackage{circuitikz}
\usetikzlibrary{matrix}
\usepackage{tkz-base}
\usetikzlibrary{calc}
\usepackage{tikz-cd}
\usetikzlibrary{arrows}
\usetikzlibrary{calc,fadings,decorations.pathreplacing}
\usetikzlibrary{shapes.geometric}
\usepackage{slashed}
\usepackage{comment}
\usepackage{upgreek}
\usepackage{multirow}
\usepackage{mathrsfs}  
\usepackage{tocloft}
\usepackage{makecell}
\usepackage{longtable}
\usepackage{chngcntr}

\newtheorem{thm}{Theorem}[section]
\newtheorem{cor}[thm]{Corollary}
\newtheorem{lem}[thm]{Lemma}
\newtheorem{prop}[thm]{Proposition}

\theoremstyle{definition}

\newtheorem{defin}[thm]{Definition}

\newtheorem{egs}[thm]{Examples}

\counterwithin*{case}{subsubsection}
\makeatletter\@addtoreset{case}{thm}\makeatother

\cftsetindents{section}{0em}{2em}
\cftsetindents{subsection}{0em}{2em}

\usepackage{bookmark}

\allowdisplaybreaks

\numberwithin{equation}{section}

\newcommand{\define}[1]{%
	\textit{#1}%
}

\DeclareMathOperator{\Hom}{Hom}
\DeclareMathOperator{\End}{End}
\DeclareMathOperator{\Aut}{Aut}

\DeclareMathOperator{\Id}{Id}

\DeclareMathOperator{\im}{Im}

\DeclareMathOperator{\Spec}{Spec}

\DeclareMathOperator{\sym}{Sym}

\DeclareMathOperator{\tr}{tr}

\DeclareMathOperator{\dvol}{dvol}

\DeclareMathOperator{\Ad}{Ad}
\DeclareMathOperator{\ad}{ad}
\DeclareMathOperator{\coker}{coker}

\DeclareMathOperator{\ind}{Index}

\DeclareMathOperator{\Cas}{Cas}

\DeclareMathOperator{\Ind}{Ind}

\DeclareMathOperator{\Ric}{Ric}
\DeclareMathOperator{\virtualdim}{virtual-dim}

\DeclareMathOperator{\ch}{ch}

\DeclareMathOperator{\Sf}{sf}

\DeclareFontFamily{U}{MnSymbolC}{}
\DeclareSymbolFont{MnSyC}{U}{MnSymbolC}{m}{n}
\DeclareFontShape{U}{MnSymbolC}{m}{n}{
	<-6>  MnSymbolC5
	<6-7>  MnSymbolC6
	<7-8>  MnSymbolC7
	<8-9>  MnSymbolC8
	<9-10> MnSymbolC9
	<10-12> MnSymbolC10
	<12->   MnSymbolC12}{}
\DeclareMathSymbol{\intprod}{\mathbin}{MnSyC}{'270}

\begin{document}
\begin{titlepage}
	\vskip 2.0cm

\begin{center}

	{\Large\bf
		Deformation Theory of Asymptotically Conical $Spin(7)$-Instantons
	}
			
		\vspace{8mm}
			
		{\large Tathagata~Ghosh
			}
			\\[4mm]
			\noindent {\em School of Mathematics, University of Leeds,\\ Woodhouse Lane, Leeds, LS2 9JT, UK}\\
			{Email: \href{mailto:T.Ghosh@leeds.ac.uk}{T.Ghosh@leeds.ac.uk}}			
			\vspace{8mm}

\begin{abstract}
	We develop the deformation theory of instantons on asymptotically conical $Spin(7)$-manifolds where the instanton is asymptotic to a fixed nearly $G_2$-instanton at infinity. By relating the deformation complex with  spinors, we identify the space of infinitesimal deformations with the kernel of the twisted negative Dirac operator on the asymptotically conical $Spin(7)$-manifold.\par 
	Finally we apply this theory to describe the deformations of Fairlie--Nuyts--Fubini--Nicolai $Spin(7)$-instantons on $\mathbb{R}^8$, where $\mathbb{R}^8$ is considered to be an asymptotically conical $Spin(7)$-manifold asymptotic to the cone over $S^7$. We calculate the virtual dimension of the moduli space using Atiyah--Patodi--Singer index theorem and the spectrum of the twisted Dirac operator.
\end{abstract}

\end{center}

\tableofcontents

\end{titlepage}


\section{Introduction}\label{section1}
Instantons on $4$-manifolds are connections whose curvatures are anti-self-dual. Instantons solve the Yang--Mills equation and hence have always been of interest to physicists in the contexts of quantum field theory, string theory, M-theory, supergravity etc. Instantons in dimensions higher than $4$ were also studied by many physicists before Donaldson--Thomas \cite{DT1998} and Donaldson--Segal \cite{donaldson2009gauge} explained their importance and scope to mathematical audience. Analogous to the $4$-dimensional case, their prediction of the possibility to construct invariants from the moduli space has been one of the main sources of motivation behind the research on higher dimensional gauge theory for mathematicians.\par
The $Spin(7)$-instantons are instantons on $8$-dimensional manifolds with $Spin(7)$-structures. The $Spin(7)$-instanton equation appeared in various physics literature; Fairlie--Nuyts \cite{fairlie1984paper} and Fubini--Nicolai \cite{fubini1985paper} have discussed $Spin(7)$-instantons on $\mathbb{R}^8$. Donaldson--Thomas \cite{DT1998} and Carrion \cite{carrion1998generalization} have discussed $Spin(7)$-instantons more generally, and around the same time, in 1998, Lewis also discussed $Spin(7)$-instantons in his PhD thesis \cite{lewis1998spin7}. In recent years, $Spin(7)$-instantons have been studied by S\'{a} Earp \cite{sa2009instantons}, Tanaka \cite{tanaka2012paper}, Walpuski \cite{walpuski2009phd}, Lotay--Madsen \cite{Lotey2022paper}.\par 
In this paper, we develop the deformation theory of instantons on a particular type of non-compact $Spin(7)$-manifolds known as asymptotically conical $Spin(7)$-manifolds. These manifolds are complete $Spin(7)$-manifolds asymptotic to the cone over compact nearly $G_2$-manifolds. The instantons on these manifolds also exhibit the asymptotically conical behaviour. Assuming that the instanton is unobstructed, we prove that the moduli space of these instantons is a manifold, and describe a way to calculate the (virtual) dimension. In the second part of the paper, we apply the deformation theory on certain instantons on $\mathbb{R}^8$, first constructed by Fairlie--Nuyts \cite{fairlie1984paper} and Fubini--Nicolai \cite{fubini1985paper} independently, in the context of supergravity. $\mathbb{R}^8$ is indeed an asymptotically conical $Spin(7)$-manifold, and hence it is appropriate to study the deformations of these instantons using our theory. The main result for this part is the calculation of the virtual dimension of the moduli space of these instantons.\par 
The study of asymptotically conical $Spin(7)$-manifolds goes back to 1989, when Bryant--Salamon \cite{bryant1989construction} gave an example of a complete non-compact $Spin(7)$-manifold, namely, the negative spinor bundle over the $4$-sphere. In 2014, Clarke \cite{clarke2014instantons} constructed a $Spin(7)$-instanton on this Bryant--Salamon $Spin(7)$-manifold. The manifolds $\mathbb{R}^8$, the $8$-dimensional Euclidean space, and $\slashed{S}^-(S^4)$, the negative spinor bundle over the $4$-sphere, are both examples of asymptotically conical $Spin(7)$-manifolds. $\mathbb{R}^8$ is asymptotic to the cone over $S^7$ with standard metric and $\slashed{S}^-(S^4)$ is asymptotic to the cone over $S^7$ with squashed metric, where both $S^7$ with standard metric and $S^7$ with squashed metric are examples of nearly $G_2$-manifolds. Recently, Foscolo \cite{fascolo2021} has constructed infinitely many examples of $Spin(7)$-manifolds of ALC type.\par 
Asymptotically conical manifolds have been studied by many authors, e.g., asymptotically conical $G_2$ manifolds by Karigiannis--Lotay \cite{karigiannis2020deformation} and recently, asymptotically conical $Spin(7)$ manifolds were studied by Lehmann \cite{Lehmann2021paper}. The analytic frameworks for studying asymptotically conical manifolds, namely, the weighted Sobolev theory and theory of asymptotically conical Fredholm and elliptic operators, have been developed by Lockhart--McOwen \cite{lockhart1985elliptic} and Marshall \cite{marshal2002deformations}.\par
Our work on deformation theory in dimension $8$ has been partially motivated by similar work in dimension $7$, namely the deformation theory of asymptotically conical $G_2$-instantons, developed by Driscoll \cite{driscoll2020thesis}, utilising the works of Harland--Ivanova--Lechtenfeld--Popov \cite{harland2010paper}, Charbonneau--Harland \cite{charbonneau2016deformations}. Asymptotically conical $G_2$-manifolds are asymptotic to the cone over nearly K\"ahler manifolds. Instantons on asymptotically conical $G_2$-manifolds have also been studied by Clarke \cite{clarke2014instantons}, Oliveira \cite{oliveira2014monopoles}, Lotay--Oliveira \cite{lotay2018su2insantons}.\par
Here is a brief outline of this paper.\par  
After we discuss the basic notations and definitions, and fix conventions related to asymptotically conical $Spin(7)$-instantons in Section \ref{section2}, we develop the deformation theory of asymptotically conical $Spin(7)$-instantons in Section \ref{section3}. In the first part, we discuss the analytical framework to study instantons of asymptotically conical $Spin(7)$-manifolds. We use Lockhart--McOwen theory, and the relation between the Dirac operator on the cone and the Dirac operator on the link to show that the Dirac operator on the asymptotically conical manifold is Fredholm only when the rate of decay is not a critical weight, and the critical weights are precisely the rates that differ from the eigenvalues of the Dirac operator on the link by a fixed constant.\par 
In the second part of Section \ref{section3}, using the analytical framework and implicit function theorem, we prove that if the rate of decay is not a critical weight, the moduli space of asymptotically conical $Spin(7)$-instantons is a smooth manifold, given that the deformations are unobstructed; moreover the dimension of the moduli space is precisely the index of the Dirac operator on the asymptotically conical manifold.\par 
In Sections \ref{section4} and \ref{section5} we carry out an in-depth study of Fairlie--Nuyts--Fubini--Nicolai (FNFN) $Spin(7)$-instanton on $\mathbb{R}^8$ and its deformation theory. We apply the deformation theory developed in Section \ref{section3} by considering $\mathbb{R}^8$ to be the asymptotically conical $Spin(7)$-manifold asymptotic to the nearly $G_2$-manifold $S^7$.\par
In order to study the moduli space, we need to identify the critical weights, and hence need to calculate the eigenvalues of the Dirac operator on the link $S^7$ in a certain range determined by the fastest rate of convergence of FNFN-instanton. In Section \ref{section4} we use various techniques in representation theory and harmonic analysis, namely, the Frobenius reciprocity to decompose the space of $L^2$-sections of the spinor bundle into direct sums of finite dimensional Hilbert spaces indexed by $Spin(7)$-representations. Moreover, we express the Dirac operator as a sum of Casimir operators. We also calculate an eigenvalue bound which yields only six representations of $Spin(7)$ for which the eigenvalues of the Dirac operator could be in the prescribed range. Then we explicitly calculate the eigenvalues of the Dirac operator for these representations, and identify the critical rates.\par
In the first part of Section \ref{section5}, we give a short description of the FNFN $Spin(7)$-instanton. In the second part we use Atiyah--Patodi--Singer theorem and the critical rates calculated in Section \ref{section4} to calculate the virtual dimension of the moduli space of the FNFN instanton. It turns out that the virtual dimensions of the moduli space are determined by precisely two known deformations of FNFN-instanton, namely dilation and translations.

\section*{Acknowledgement}
I would like to express my heartfelt thanks to PhD supervisor Dr.~Harland for his continuous guidance on this project and spending hours discussing various ideas with me. I would also like to thank University of Leeds for giving me the opportunity to carry out my research. I would like to thank the anonymous reviewer for many helpful comments and suggestions.

\section{Preliminaries}\label{section2}
\subsection{Nearly \texorpdfstring{$G_2$}{G2}-Manifolds}
Let $\Sigma$ be a Riemannian $7$-dimensional manifold. A $3$-form $\phi \in \Omega^3(\Sigma)$ is called a \define{$G_2$-structure on $\Sigma$} if in a local orthonormal frame $e^1, \dots, e^7$, $\phi$ can be written as
	\begin{align}\label{g2-3form}
	\phi = e^{127} + e^{347} + e^{567} + e^{145} + e^{136} + e^{235} - e^{246}
\end{align} 
where $e^{ijk} := e^i \wedge e^j \wedge e^k$. The group $G_2$ is the stabilizer group of $\phi$ restricted to each tangent space, and is a $14$-dimensional simple, connected, simply connected Lie group. The group $G_2$ is also the automorphism group $\Aut(\mathbb{O})$. We denote by $\Omega^k_d(\Sigma)$ the $G_2$-invariant subspace of $\Omega^k(\Sigma)$ with point-wise dimension $d$ (see \cite{salamon2010notes}). For more details on $G_2$-structures, see \cite{bryant1985metrics}, \cite{bryant1987metrics}, \cite{bryant2003some}, \cite{karigiannis2006some}, \cite{karigiannis2009flow}, \cite{karigiannis2020book}.\par 
A Riemannian $7$-manifold possesses a $G_2$-structure if and only if it is a spin manifold \cite{lawson2016spin}. Hence, we now discuss the spinor bundle. For our purpose, we start by fixing a representation of the Clifford algebra $Cl(7)$ in which the volume form $\Gamma_7$ acts as $-\Id$. Let $\slashed{S}(\Sigma)$ be the spinor bundle over a $7$-manifold $\Sigma$ with $G_2$-structure. Let $\xi \in \Gamma(\slashed{S}(\Sigma))$ be a unit spinor such that $\omega \cdot \xi = 0$ for all $\omega \in \Omega^2_{14}$, where $\cdot$ denotes Clifford multiplication. Then we have an isomorphism
	\begin{align}\label{spinor-iso}
		s : \Lambda^0(T^*\Sigma) \oplus \Lambda^1(T^*\Sigma) &\to \slashed{S}(\Sigma)\nonumber\\
		(f, v) &\mapsto (f-v)\cdot \xi.
	\end{align}
The proofs of the following results are similar to the proofs given by \cite{driscoll2020paper} or \cite{singhal2021paper}.
\begin{lem}\label{eigenphi}
	The $3$-form $\phi$ acts on the subspaces $\Lambda^0$ and $\Lambda^1$ of $\slashed{S}(\Sigma)$ with eigenvalues $7$ and $-1$ respectively, whereas the $4$-form $\psi = *\phi$ acts on $\Lambda^0$ and $\Lambda^1$ with eigenvalues $-7$ and $1$ respectively.
\end{lem}
\begin{cor}
	Let $f \in \Omega^0(\Sigma), v, u \in \Omega^1(\Sigma)$. Then Clifford multiplication of $(f, v)$ by $u$ is
	\begin{equation}\label{cliffmult1}
    u \cdot(f, v) = (\langle u, v\rangle, - fu - (u\wedge v)\intprod\phi).
\end{equation}
\end{cor}
\begin{defin}
    Let $\Sigma$ be a $7$-dimensional Riemannian manifold and $\phi \in \Omega^3(\Sigma)$. Then $\phi$ is called a \define{nearly (parallel) $G_2$-structure} on $\Sigma$ if it satisfies
    \begin{equation}
        d\phi = \tau_0\psi,
    \end{equation}
    where $\psi = *\phi$ and $\tau_0 \in \mathbb{R} \setminus\{0\}$. In this case, $(\Sigma, \phi)$ is called a \define{nearly $G_2$-manifold}.
\end{defin}
Clearly $\phi$ is not closed, but is co-closed. For more on nearly $G_2$-structures, see \cite{alexandrov2012}, \cite{friedrich1997paper}.
\begin{defin}
    Let $\slashed{S}(\Sigma)$ be the spinor bundle on $\Sigma$. A real spinor $\xi \in \Gamma(\slashed{S}(\Sigma))$ is called a \define{Killing spinor} if there exists $\delta \in \mathbb{R} \setminus\{0\}$ such that for all $X \in \Gamma(T\Sigma)$, $\xi$ satisfies the \define{Killing equation} 
    \begin{equation}
        \nabla_X \xi = \delta X\cdot \xi.
    \end{equation}
    The scalar $\delta$ is called the \define{Killing constant} for the Killing spinor $\xi$.
\end{defin}
We note that a unit spinor $\xi$ on a nearly $G_2$-manifold satisfying $\omega \cdot \xi = 0$ for all $\omega \in \Omega^2_{14}$ is a Killing spinor. Conversely, any Riemannian $7$-manifold admitting a Killing spinor is a nearly $G_2$-manifold. In fact, there is a one to one correspondence between nearly $G_2$-structures and real Killing spinors on $\Sigma$ \cite{baum1990twistor}. A nearly $G_2$-structure $\phi$ satisfying $d\phi = \tau_0\psi$ corresponds to $\xi \in \slashed{S}(\Sigma)$ such that $\nabla_X \xi = \frac{\tau_0}{8}X\cdot \xi$, where we have used the fact that the Killing constant $\delta$ can be written in term of $\tau_0$ as $\delta = \frac{\tau_0}{8}$. If $g$ is the metric induced by the nearly $G_2$ structure $\phi$, then the Ricci curvature is given by $\Ric_g = \frac{3}{8}\tau_0^2g$, and hence every nearly $G_2$ manifold is Einstein.\par
We note that we can always re-scale $\tau_0$. We take $\tau_0 = 4$, then we have $d\phi = 4\psi$ and $\nabla_X \xi = \frac{1}{2}X\cdot \xi$. For a nearly $G_2$-manifold $(\Sigma, \phi)$, we can define a $1$-parameter family of affine connections on $T\Sigma$. Let $t \in \mathbb{R}$. Then $\nabla^t$ is a $1$-parameter family of connection on $T\Sigma$ defined by
\begin{equation}
    g(\nabla^t_XY, Z) = g(\nabla_XY, Z)+\frac{t}{3}\phi(X, Y, Z)
\end{equation}
for $X, Y, Z \in \Gamma(T\Sigma)$. The totally skew-symmetric torsion tensor $T^t$ is given by
\begin{equation}\label{torsion}
    T^t(X, Y) = \frac{2t}{3}\phi(X, Y, \cdot)
\end{equation}
Now, lifting $\nabla^t$ to the spinor bundle $\slashed{S}(\Sigma)$ and using the eigenvalues (\ref{eigenphi}), we find
\begin{equation}
    \nabla^t_X\xi = -\frac{t-1}{2}X \cdot \xi,
\end{equation}
where $\xi \in \Gamma(\slashed{S}(\Sigma))$ is a Killing spinor and $X \in \Gamma(T\Sigma)$. Therefore, for $t = 1$, the Killing spinor $\xi$ is parallel with respect to the connection $\nabla^1$. Then the connection $\nabla^1$ has holonomy group contained in $G_2$ with totally skew-symmetric torsion. This connection $\nabla^1$ on the nearly $G_2$-manifold $\Sigma$ is known as the \define{canonical connection}.

\subsection{\texorpdfstring{$Spin(7)$}{Spin(7)}-Manifolds and \texorpdfstring{$Spin(7)$}{Spin(7)}-Instantons}
Let $X$ be an $8$-dimensional Riemannian manifold equipped with a $4$-form $\Phi \in \Omega^4(X)$ such that in local orthonormal basis $e^0, e^1, \dots, e^7$, we have $\Phi = e^0 \wedge \phi + \psi$ where $\phi$ is as in (\ref{g2-3form}) and $*(e^0\wedge\phi) = \psi$. Then $\Phi$ is said to be a \define{$Spin(7)$-structure} on $X$ and $(X, \Phi)$ is said to be an \define{almost $Spin(7)$-manifold}. If $\Phi$ is torsion-free, i.e., if $\nabla\Phi = 0$ where $\nabla$ is the Levi--Civita connection, or equivalently, if $d\Phi = 0$, then $(X, \Phi)$ is called a \define{$Spin(7)$-manifold}.\par
The $21$-dimensional Lie group $Spin(7)$ is the stabiliser group of $\Phi$ restricted to each tangent space. We denote by $\Omega^k_d(X)$ the $Spin(7)$-invariant subspace of $\Omega^k(X)$ with point-wise dimension $d$ (see \cite{salamon2010notes}). For more details on $Spin(7)$-manifolds, see \cite{bryant1985metrics}, \cite{karigiannis2006some}.\par 
If $\Phi$ is a $Spin(7)$-structure on a manifold $X$, then $X$ is a Spin manifold. Moreover, if $\Phi$ is torsion-free, then $X$ admits a non-trivial parallel spinor \cite{joyce2000compact}. The canonical spin structure can be identified in the following way.
$$\slashed{S}^+ \cong \Lambda^0 \oplus \Lambda^2_7\ \ \text{and }\ \ \slashed{S}^- \cong \Lambda^1_8.$$
Let $X$ be a $Spin(7)$ manifold and $P$ be a principal $G$-bundle on $X$. Let $\mathfrak{g}_P$ be the adjoint vector bundle. Then we have the decomposition $\Omega^2(\mathfrak{g}_P) = \Omega^2_7(\mathfrak{g}_P) \oplus \Omega^2_{21}(\mathfrak{g}_P)$.
\begin{defin}
	Let $\pi_7^2 : \Omega^2(\mathfrak{g}_P) \to \Omega^2_7(\mathfrak{g}_P)$ be the projection. Then a connection $A$ on $P$ is said to be a \define{$Spin(7)$-instanton} if $\pi_7^2(F_A) = 0$ where $F_A$ is the curvature of the connection $A$. In this case $F_A \in \Omega^2_{21}(\mathfrak{g}_P)$. Equivalently, $A$ is a $Spin(7)$ instanton if it satisfies
    \begin{equation}
       *(\Phi \wedge F_A) = -F_A.
     \end{equation}
\end{defin}
This follows from the fact that the operator on $\Lambda^2 \cong \mathfrak{so}(8)$ defined by $\omega \mapsto *(\Phi \wedge \omega)$ has eigenvalues $-1$ and $3$ with eigenspaces $\Lambda^2_{21} \cong \mathfrak{spin}(7)$ and $\Lambda^2_7$ respectively.\par 
Moreover, $A$ is an instanton if and only if $F_A$ annihilates the parallel spinor, i.e., we have $F_A \cdot \xi = 0$, where $\cdot$ denotes Clifford multiplication.

\subsection{Asymptotically Conical \texorpdfstring{$Spin(7)$}{Spin(7)}-Manifolds}
Let $(\Sigma, g_\Sigma)$ be a Riemannian $7$-manifold with a nearly $G_2$-structure $\phi$ satisfying $d\phi = 4\psi$ where $\psi = *\phi$. A \define{$Spin(7)$-cone} on $\Sigma$ is $C(\Sigma) := (0, \infty) \times \Sigma$ together with a $Spin(7)$-structure $(C(\Sigma), \Phi_C)$ defined by
\begin{equation}
    \Phi_C := r^3dr \wedge \phi + r^4 \psi
\end{equation}
where $r \in (0, \infty)$ is the coordinate. $\Sigma$ is called the \define{link} of the cone. The metric $g_C$ compatible with $\Phi_C$ is given by
\begin{equation}
    g_C = dr^2 + r^2g_\Sigma.
\end{equation}
We note that condition $d\phi = 4\psi$ implies the torsion free condition $d\Phi_C = 0$, which implies that $(C(\Sigma), g_C, \Phi)$ is a $Spin(7)$-manifold.\par
We note that as a nearly $G_2$-manifold is Einstein with positive scalar curvature, $\Sigma$ is compact. We also note that a $Spin(7)$-cone is not complete. Hence, we consider complete $Spin(7)$-manifolds whose geometry is asymptotic to the given (incomplete) $G_2$-cone. 
\begin{defin}
	Let $(X, g, \Phi)$ be a $Spin(7)$-manifold. $X$ is called an \define{asymptotically conical (AC) $Spin(7)$-manifold with rate $\nu < 0$} if there exists a compact subset $K \subset X$, a compact connected nearly $G_2$ manifold $\Sigma$, and a constant $R > 1$ together with a diffeomorphism
    \begin{equation}
        h : (R, \infty) \times \Sigma \to X \setminus K
    \end{equation}
	such that
    \begin{equation}
        \left|\nabla_C^j(h^*(\Phi|_{X\setminus K}) - \Phi_C)\right|(r, p) = O(r^{\nu - j})\ \ \text{ as }r \to \infty
    \end{equation}
	for each $p \in \Sigma$, $j  \in \mathbb{Z}_{\geq 0}$, $r \in (R, \infty)$; where $\nabla_C$ is the Levi--Civita connection for the cone metric $g_C$ on $C(\Sigma)$, and the norm is induced by the metric $g_C$.\par
	$X\setminus K$ is called the \define{end} of $X$ and $\Sigma$ the \define{asymptotic link} of $X$.
\end{defin}
It can be proved that (see \cite{karigiannis2020deformation}) the metric $g$ satisfies the same asymptotic condition
$$\left|\nabla_C^j(h^*(g|_{X\setminus K}) - g_C)\right| = O(r^{\nu - j})\ \ \text{ as }r \to \infty$$
\begin{egs}
The simplest example of an AC $Spin(7)$ manifold is $(\mathbb{R}^8, \Phi_0)$. Since $C(S^7) = \mathbb{R}^8\setminus\{0\}$, $(\mathbb{R}^8, \Phi_0)$ is an AC manifold with any rate $\nu < 0$. Bryant--Salamon $Spin(7)$-manifold $\slashed{S}^-(S^4)$ is another example of an AC $Spin(7)$ manifold. This a rank $4$ bundle, hence the total space is a manifold of dimension $8$. This is an AC $Spin(7)$-manifold asymptotic to the cone over $(S^7$ with squashed metric with rate $\nu = -10/3$ (see \cite{bryant1989construction}). Some more examples of AC $Spin(7)$-metrics can be found in the recent work of Lehmann \cite{Lehmann2020paper}.

\end{egs}
\subsection{Lockhart--McOwen Analysis on AC \texorpdfstring{$Spin(7)$}{Spin(7)}-Manifold}
Now we review Lockhart--McOwen analysis applied to AC $Spin(7)$ manifolds.\par 
Let $X$ be an AC $Spin(7)$ manifold. In order to define ``weighted Banach spaces'' on $X$, we first define a notion of radius function.
\begin{defin}
	A \define{radius function} is a map $\varrho : X \to \mathbb{R}$ defined by
 \begin{equation}\label{radiusfn}
     \varrho(x) := \begin{cases}
		1 &\text{ if $x \in $ the compact subset $K \subset X$}\\
		r &\text{ if $x = h(r, p)$ for some $r \in (2R, \infty)$, $p \in \Sigma$}\\
		\widetilde{r} &\text{ if $x \in h((R, 2R) \times \Sigma)$}
	\end{cases}
 \end{equation}
	where $h : (R, \infty) \times \Sigma \to X \setminus K$ is the diffeomorphism, and $\widetilde{r}$ is a smooth interpolation between its definition at infinity and its definition on $K$, in a decreasing manner.
\end{defin}
Let $\pi : E \to X$ be a vector bundle over $X$ with a fibre-wise metric and a connection $\nabla$ compatible with the metric.
\begin{defin}
	Let $p \geq 1, k \in \mathbb{Z}_{\geq 0}, \nu \in \mathbb{R}$ and $C_c^\infty(E)$ be the space of compactly supported smooth sections of $E$. We define the \define{conically damped} or \define{weighted Sobolev space} $W^{k,p}_\nu(E)$ of sections of $E$ over $X$ of weight $\nu$ as follows:\\
	For $\xi \in C_c^\infty(E)$, we define the \define{weighted Sobolev norm} $\|\cdot\|_{W^{k,p}_\nu(E)}$ as
    \begin{equation}\label{wetsobnorm}
        \|\xi\|_{W^{k,p}_\nu(E)} = \left(\sum\limits_{j=0}^k\int_X\left|\varrho^{-\nu+j}\nabla^j\xi\right|^p\varrho^{-8}\dvol\right)^{1/p}
    \end{equation}
	which is clearly finite and indeed a norm. Then the weighted Sobolev space $W^{k,p}_\nu(E)$ is the completion of $C_c^\infty(E)$ with respect to the norm $\|\cdot\|_{W^{k,p}_\nu(E)}$.
\end{defin}
We note that $W^{0,2}_{-4}(E) = L^2(E)$. Moreover, we have $\varrho^\nu W^{0,2}_\mu(E) = W^{0,2}_{\mu+\nu}(E)$. In particular, $W^{0,2}_\nu(E) = \varrho^{4+\nu}L^2(E)$ (see \cite{Lehmann2021paper}).
\begin{defin}
    Let $k \in \mathbb{Z}_{\geq 0}$ and $\nu \in \mathbb{R}$. Then for $\xi \in C_c^\infty(E)$, we define the \define{weighted $C^k$ norm} $\|\cdot\|_{C^k_\nu(E)}$ as
    \begin{equation}
        \|\xi\|_{C^k_\nu(E)} = \sum\limits_{j=0}^k\|\varrho^{-\nu+j}\nabla^j\xi\|_{C^0}
    \end{equation}
    which is well defined and a norm. Then the \define{weighted $C^k$ space} $C^k_\nu(E)$ is the closure of $C_c^\infty(E)$ with respect to this norm. We also define $C^\infty_\nu(E) := \bigcap\limits_{k \geq 0}C^k_\nu(E)$.
\end{defin}
\begin{thm}[\define{Weighted Sobolev Embedding Theorem}]\cite{marshal2002deformations}
\leavevmode
	\begin{enumerate}
	\item Let $k, l \geq 0$. If $k - \frac{8}{p} \geq l$, then $W^{k,p}_\nu(E) \hookrightarrow C^l_\nu(E)$ is a continuous embedding.
	 \item Let $k \geq l \geq 0$, $p \leq q$ and $\mu \leq \nu$. If $k - \frac{8}{p} \geq l- \frac{8}{q}$, then $W^{k,p}_\mu(E) \hookrightarrow W^{l,q}_\nu(E)$ is a continuous embedding.
	\end{enumerate}
\end{thm}
In order to ensure that we work with continuous sections, we shall always assume $k \geq 4$. This follows from the first part of the weighted Sobolev embedding theorem by putting $l = 0$ and $p = 2$.
\begin{thm}[\define{Weighted Sobolev Multiplication Theorem}]\cite{driscoll2020thesis}
	Let $\xi \in W^{k,2}_\mu(E), \eta \in W^{l,2}_\nu(F)$. If $l \geq k > \frac{8}{2}=4$, then the multiplication
	$$W^{k,2}_\mu(E) \times W^{l,2}_\nu(F) \to W^{k,2}_{\mu+\nu}(E \otimes F)$$
	is bounded. In other words, there is a constant $C > 0$ such that
	$$\|\xi \otimes \eta\|_{W^{k,2}_{\mu+\nu}(E \otimes F)} \leq C\|\xi\|_{W^{k,2}_\mu(E)}\|\eta\|_{W^{l,2}_\nu(F)}.$$
\end{thm}
\begin{prop}\cite{Lehmann2021paper}\label{pairing}
    Let $\xi \in W^{0,2}_\mu(E), \eta \in W^{0,2}_\nu(E)$. If $\mu + \nu < 8$, then $\langle \xi, \eta\rangle_{L^2} = \int_M \langle \xi, \eta\rangle \dvol$ is finite and satisfies,
    $$\langle \xi, \eta\rangle_{L^2} \leq \|\xi\|_{W^{0,2}_\mu(E)}\|\eta\|_{W^{0,2}_\nu(E)}.$$
\end{prop}
From Proposition \ref{pairing}, we have the pairing $\langle \cdot\ , \cdot \rangle_{L^2} : W^{0,2}_\nu(E) \times W^{0,2}_{-8-\nu}(E) \to \mathbb{R}$. This defines the isomorphism \cite{Lehmann2021paper}
\begin{equation}
    \left(W^{0,2}_\nu(E)\right)^* \cong W^{0,2}_{-8-\nu}(E).
\end{equation}\par
Now, let $P \to X$ be a principal $G$ bundle. Consider the associated vector bundle $E := T \otimes \mathfrak{g}_P$, where $T$ is either $\Lambda^kT^*X$ of $k$-forms, or $\Lambda^*T^*X = \bigoplus\limits_{k=0}^8\Lambda^kT^*X$, or the spinor bundle $\slashed{S}$ over $X$. If $A$ is a connection on $\mathfrak{g}_P$, then $E$ inherits a metric from $T$ and a connection from the Levi--Civita connection on $T$ and the connection on $\mathfrak{g}_P$. If $\xi \in C_c^\infty(T \otimes \mathfrak{g}_P) \subset \Gamma(T \otimes \mathfrak{g}_P)$, then the weighted Sobolev norm is given by Equation \ref{wetsobnorm} where $\nabla = \nabla^{LC} \otimes 1_{\mathfrak{g}_P} + 1_T \otimes \nabla^A$, $\nabla^A$ being the connection on $\mathfrak{g}_P$.\par
Before moving forward let us fix few notations:
\begin{align*}
	\Omega^{m,k}_\nu(X) &:= W^{2,k}_\nu(\Lambda^mT^*X),\ \ \Omega^{m,k}_\nu(\mathfrak{g}_P) := W^{2,k}_\nu(\Lambda^mT^*X \otimes \mathfrak{g}_P)\\
	\Omega^{*,k}_\nu(X) &:= W^{2,k}_\nu(\Lambda^*T^*X),\ \ \ \ \Omega^{*,k}_\nu(\mathfrak{g}_P) := W^{2,k}_\nu(\Lambda^*T^*X \otimes \mathfrak{g}_P).	
\end{align*}
\subsection{Asymptotically Conical \texorpdfstring{$Spin(7)$}{Spin(7)}-Instantons and Moduli Space}
\begin{defin}
	Let $X$ be an AC $Spin(7)$-manifold asymptotic to the cone $C(\Sigma)$. Let $P \to X$ be a principal $G$-bundle over $X$. $P$ is called \define{asymptotically framed} if there exists a principal bundle $Q \to \Sigma$ such that $h^*P \cong \pi^*Q$, where $\pi : C(\Sigma) \to \Sigma$ is the natural projection.
\end{defin}
We note that such framing always exists \cite{lotay2018su2insantons}. So we fix a framing $Q$.
\begin{defin}
	Let $X$ be an AC $Spin(7)$-manifold asymptotic to the cone $C(\Sigma)$. Let $P \to X$ be an asymptotically framed bundle. A connection $A$ on $P$ is called an \define{asymptotically conical connection} with rate $\nu$ if there exists a connection $A_\Sigma$ on $Q \to \Sigma$ such that
	\begin{equation}\label{AC-ins}
		\left|\nabla_C^j(h^*(A) - \pi^*(A_\Sigma))\right| = O(r^{\nu - 1 - j})\ \ \text{ as }r \to \infty
	\end{equation}
	for each $p \in \Sigma$, $j \in \mathbb{Z}_{\geq 0}$, $\nu < 0$. The norm is induced by the cone metric and the metric on $\mathfrak{g}$.\par 
	$A$ is called \define{asymptotic} to $A_\Sigma$ and $\nu_0 := \inf\{\nu : A \text{ is AC with rate }\nu\}$ is called the \define{fastest rate of convergence of $A$}.
\end{defin}
Let $\mathcal{A}_P$ be the space of AC connections on $P$. Fix a reference connection $A \in \mathcal{A}_P$. Then for any other connection $A' = A + \alpha$, $\alpha \in \Omega^1(\mathfrak{g}_P)$ identifies the spaces $\mathcal{A}_P$ and $\Omega^1(\mathfrak{g}_P)$. Denote the space of $W^{k,2}_{\nu-1}$-connections by
\begin{equation}
    \mathcal{A}_{k,\nu-1} := \{A + \alpha : \alpha \in \Omega^{1,k}_{\nu-1}(\mathfrak{g}_P)\}
\end{equation}
and define
\begin{equation}
    \mathcal{A}_{\nu-1} := \bigcap\limits_{k=1}^\infty\mathcal{A}_{k,\nu-1}
\end{equation}
which is the space of $C^\infty_{\nu-1}$-connections.\par
Now, a gauge transform is $\varphi \in \Aut(P)$ and acts on a connection $A$ by $\varphi \cdot A = \varphi A \varphi^{-1} - d\varphi\varphi^{-1}$. Let $G \to GL(V)$ be a faithful representation of $G$, and consider the associated vector bundle $\End(V)$. Then we define the \define{weighted gauge group} by (see \cite{nakajima1990moduli}) 
\begin{equation}\label{gaugetrans}
    \mathcal{G}_{k+1,\nu} := \{\varphi \in C^0(\End(V)) : \|\varphi - I\|_{k+1,\nu} < \infty, \varphi \in G\}.
\end{equation}
We also define $\mathcal{G}_\nu := \bigcap\limits_{l=1}^\infty\mathcal{G}_{l,\nu}$.
\begin{lem}\label{acmodu}\cite{donaldson2002floerbook}
	The point-wise exponential map defines charts for which $\mathcal{G}_{k+1,\nu}$ is a Hilbert Lie group with Lie algebra modelled on $\Omega^{0,k+1}_\nu(\mathfrak{g}_P)$ for $k \geq 3$. The group $\mathcal{G}_{k+1,\nu}$ acts on $\mathcal{A}_{k,\nu-1}$ smoothly via gauge transformations, for $k \geq 4$.
\end{lem}
\begin{defin}
	Let $X$ be an AC $Spin(7)$-manifold asymptotic to $C(\Sigma)$. Let $P \to X$ be a principal $G$-bundle asymptotically framed by $Q \to \Sigma$. Let $A_\Sigma$ be an instanton on the nearly $G_2$ manifold $\Sigma$. Then the \define{moduli space of $Spin(7)$-instantons asymptotic to $A_\Sigma$ with rate $\nu$} is given by
	\begin{equation}
		\mathcal{M}(A_\Sigma, \nu) := \{Spin(7) \text{ instanton $A$ on $P$ satisfying (\ref{AC-ins}) asymptotic to }A_\Sigma\}/\mathcal{G}_\nu.
	\end{equation}
\end{defin}
\section{Deformation Theory of Asymptotically Conical \texorpdfstring{$Spin(7)$}{Spin(7)}-Instantons}\label{section3}
In this section we describe the deformation theory of asymptotically conical $Spin(7)$-instantons. In the first part we discuss the necessary analytic framework, following the works of Lockhart--McOwen \cite{lockhart1985elliptic}, Marshall \cite{marshal2002deformations}, Karigiannis--Lotay \cite{karigiannis2020deformation} and Driscoll \cite{driscoll2020paper}. In the second part we develop the general theory, where we closely follow Donaldson \cite{donaldson1990geometry} and Driscoll \cite{driscoll2020paper}.
\subsection{Fredholm and Elliptic Asymptotically Conical Operators}
We begin this section by defining the operators that will be important in developing the deformation theory.
Let $\Sigma$ be a nearly $G_2$-manifold and $Q \to \Sigma$ be a principal $G$-bundle. Let $A_\Sigma$ be a connection on $Q$. Consider the bundle $\slashed{S}(\Sigma) \otimes \mathfrak{g}_Q$ where $\slashed{S}(\Sigma)$ is the spinor bundle on $\Sigma$ and $\mathfrak{g}_Q = Q \times_{\Ad}\mathfrak{g}$. Then we have a twisted Dirac operator
    \begin{equation}\label{diracopsig}
        \slashed{\mathfrak{D}}_{A_\Sigma} : \Gamma(\slashed{S}(\Sigma) \otimes \mathfrak{g}_Q) \to \Gamma(\slashed{S}(\Sigma) \otimes \mathfrak{g}_Q).
    \end{equation}
Let $X$ be an AC $Spin(7)$-manifold with link $\Sigma$. Let $P \to X$ be an asymptotically framed bundle. Let $A \in \mathcal{A}_P$ be an AC connection asymptotic to $A_\Sigma$.\par 
Consider the bundle $\slashed{S}(X) \otimes \mathfrak{g}_P$ where $\slashed{S}(X)$ is the spinor bundle over $X$ and $\mathfrak{g}_P = P \times_{\Ad}\mathfrak{g}$. Then we have a Dirac operator 
\begin{equation}\label{diracopx}
       \slashed{\mathfrak{D}}_A : \Gamma(\slashed{S}(X) \otimes \mathfrak{g}_P) \to \Gamma(\slashed{S}(X) \otimes \mathfrak{g}_P).
\end{equation}
Let $C(\Sigma) = (0, \infty) \times \Sigma$ be a $Spin(7)$-cone over $\Sigma$ and $\pi^*Q \to C(\Sigma)$ be a principal bundle over $C(\Sigma)$. Let $A_C = \pi^*A_\Sigma$.\par 
Now consider the bundle $\slashed{S}(C(\Sigma)) \times \mathfrak{g}_{\pi^*Q}$, where $\slashed{S}(C(\Sigma))$ is the spinor bundle on $C(\Sigma)$ and $\mathfrak{g}_{\pi^*Q} = \pi^*Q \times_{\Ad} \mathfrak{g}$. Then we have a Dirac operator
\begin{equation}\label{diracopcone}
	\slashed{\mathfrak{D}}_{A_C} : \Gamma(\slashed{S}(C(\Sigma)) \otimes \mathfrak{g}_{\pi^*Q}) \to \Gamma(\slashed{S}(C(\Sigma)) \otimes \mathfrak{g}_{\pi^*Q}).
\end{equation}	
The objective behind introducing this Dirac operator $\slashed{\mathfrak{D}}_{A_C}$ is to study the Fredholm properties of the Dirac operator $\slashed{\mathfrak{D}}_A$ using Lockhart--McOwen theory.
\begin{defin}
        Let $E = \slashed{S}(C(\Sigma)) \otimes \mathfrak{g}_{\pi^*Q}$ or $\mathfrak{g}_{\pi^*Q}$ be a bundle over $C(\Sigma)$. Then a section $\sigma$ of the bundle $E$ is called \define{homogeneous of degree $\lambda$} if $\sigma = r^\lambda \eta_\Sigma$, where $\eta_\Sigma$ is a section of $\slashed{S}(\Sigma) \otimes \mathfrak{g}_Q$ or $\mathfrak{g}_Q$ respectively, lifted to the cone.
\end{defin}
Now, let $K_C$ be either the Dirac operator
$\slashed{\mathfrak{D}}_{A_C}$ or the coupled Laplace operator $d_{A_C}^*d_{A_C}$. Then we consider the set of critical weights $\mathscr{D}(K_C)$ given by
\begin{equation}
    \mathscr{D}(K_C) = \{\lambda \in \mathbb{R} : \sigma \in \Gamma(E) \text{ is non-zero of homogeneous order $\lambda$ such that }K_C(\sigma) = 0\}
\end{equation}
\begin{thm}\cite{marshal2002deformations}
	The extended map
	$$K : W^{k+l,p}_\nu(E) \to W^{k,p}_{\nu-\gamma}(F)$$
	where $K = \slashed{\mathfrak{D}}_A$ or $d_A^*d_A$, $E = F = \slashed{S}(X) \otimes \mathfrak{g}_P$, or $\mathfrak{g}_P$ and $l = \gamma$ is $1$ or $2$ respectively, is Fredholm if $\nu \in \mathbb{R} \setminus \mathscr{D}(K_C)$. Moreover, if $[\nu, \nu'] \cap \mathscr{D}(K_C) = \varnothing$, then the kernel of $K$ is independent of the weight.
\end{thm}
Hence, we focus on finding the set of critical weights for the operators $\slashed{\mathfrak{D}}_{A_C}$ and $d_{A_C}^*d_{A_C}$.
\subsubsection*{The set of critical weights for the Laplace operator $d_{A_C}^*d_{A_C}$}
We want to find the set of critical weights for the Laplace operator $d_{A_C}^*d_{A_C}$, i.e. the set $\mathscr{D}(d_{A_C}^*d_{A_C})$. This set corresponds to a subset of the kernel of the operator containing elements of homogeneous order $\lambda$. Thus, if $\lambda \in \mathscr{D}(d_{A_C}^*d_{A_C})$, then $\sigma \in \ker (d_{A_C}^*d_{A_C})$ such that $\sigma = r^\lambda\eta$ for $\eta \in \Omega^0(\mathfrak{g}_P)$. An easy calculation yields, if $\sigma = r^\lambda\eta$ for $\eta \in \Omega^0(\mathfrak{g}_P)$, then,
$$d_{A_C}^*d_{A_C}\sigma = r^{\lambda-2}(d_{A_\Sigma}^*d_{A_\Sigma}\eta - \lambda(\lambda+6)\eta).$$
Thus, $\xi$ is in the kernel if and only if $\lambda(\lambda+6)$ is an eigenvalue of $d_{A_\Sigma}^*d_{A_\Sigma}$. But since the coupled Laplace operator is positive, $(-6,0) \cap \mathscr{D}(d_{A_C}^*d_{A_C}) = \varnothing$. Hence, we have the following proposition.
\begin{prop}\label{fredhlm}
	Let $A$ be a AC connection over an AC $Spin(7)$-manifold $X$. If $\nu \in (-6, 0)$, then the coupled Laplace operator
	$$d_A^*d_A : \Omega^{0,k+2}_\nu(\mathfrak{g}_P) \to \Omega^{0,k}_{\nu-2}(\mathfrak{g}_P)$$
	is Fredholm.
\end{prop}
\subsubsection*{The set of critical weights for the Dirac operator $\slashed{\mathfrak{D}}_{A_C}$}
Now, we want to find the set of critical weights for the Dirac operator $\slashed{\mathfrak{D}}_{A_C}$, i.e. the set $\mathscr{D}(\slashed{\mathfrak{D}}_{A_C})$. This set corresponds to a subset of the kernel of the operator containing elements of homogeneous order $\lambda$. Thus, if $\lambda \in \mathscr{D}(\slashed{\mathfrak{D}}_{A_C})$, then $\sigma \in \ker\slashed{\mathfrak{D}}_{A_C}$ such that $\sigma = r^\lambda \eta_\Sigma$ where $\eta_\Sigma$ is a spinor on $\Sigma$.\par 
Recall that the volume element $\Gamma_8$ acting on $\slashed{S}(C)$ satisfies $\Gamma_8^2 = 1$ and this gives an eigen-space decomposition $\slashed{S}(C) = \slashed{S}^+(C) \oplus \slashed{S}^-(C)$ corresponding to $+1$ and $-1$ eigenvalues respectively. We also have the decomposition of the Dirac operator:
$$\mathcal{D}_C^\pm : \Gamma(\slashed{S}^\pm(C)) \to \Gamma(\slashed{S}^\mp(C)).$$
The following proposition can be easily proved by slightly modifying results in \cite{lawson2016spin} or \cite{driscoll2020thesis}.
\begin{prop}
	Consider the following two twisted Dirac operators $\slashed{\mathfrak{D}}_{A_\Sigma}$ and 
	$$\slashed{\mathfrak{D}}_{A_C}^- : \Gamma(\slashed{S}^-(C(\Sigma)) \otimes \mathfrak{g}_{\pi^*Q}) \to \Gamma(\slashed{S}^+(C(\Sigma)) \otimes \mathfrak{g}_{\pi^*Q}).$$
	Then,
 \begin{equation}\label{diracrelation}
     \slashed{\mathfrak{D}}_{A_C}^- = dr \cdot \left(\frac{\partial}{\partial r} + \frac{1}{r}\left(\frac{7}{2}-\slashed{\mathfrak{D}}_{A_\Sigma}\right)\right).
 \end{equation}
\end{prop}
Finally, we have the description of the set critical weights of the twisted Dirac operator.
\begin{prop}
	Consider the Dirac operator
	\begin{equation}\label{diracopmain}
	    \slashed{\mathfrak{D}}_A^- : W^{k+1,2}_{\nu-1}(\slashed{S}^-(X) \otimes \mathfrak{g}_P) \to W^{k,2}_{\nu-2}(\slashed{S}^+(X) \otimes \mathfrak{g}_P).
	\end{equation}
	Then the set of critical weights is given by
	\begin{equation}
     \mathscr{D}(\slashed{\mathfrak{D}}_A^-) = \left\{\nu \in \mathbb{R} : \nu + \frac{5}{2} \in \Spec\slashed{\mathfrak{D}}_{A_\Sigma}\right\}.
 \end{equation}
	Thus, this Dirac operator $\slashed{\mathfrak{D}}_A^-$ is Fredholm if $\nu+\frac{5}{2} \in \mathbb{R} \setminus\Spec\slashed{\mathfrak{D}}_{A_\Sigma}$.
\end{prop}
\subsubsection*{Index of the Dirac operator $\slashed{\mathfrak{D}}_{A_C}^-$}
\begin{defin}
	For $\lambda \in \mathbb{R}$, define the space
	$$\mathcal{K}(\lambda)_C^- := \left\{\sigma \in \ker\slashed{\mathfrak{D}}_{A_C}^- : \sigma(r, p) = r^{\lambda-1}P(r, p)\right\}$$
	where $P(r, p) = \sum\limits_{j=0}^m(\log r)^j\eta_j(\sigma)$, and each $\eta_j \in \Gamma(\slashed{S}(\Sigma) \otimes \mathfrak{g}_Q)$.
\end{defin}
The following theorem is a consequence of the fact that the Dirac operator $\slashed{\mathfrak{D}}_{A_\Sigma}$ is self-adjoint. Proof of similar result can be found in \cite{driscoll2020thesis}.
\begin{prop}\label{logterm}
	If 
	\begin{equation}\label{logterm2}
	    \slashed{\mathfrak{D}}_{A_C}^-\left(r^{\lambda-1}\sum\limits_{j=0}^m(\log r)^j\eta_j(\sigma)\right) = 0,
	\end{equation}
	then $m = 0$. Hence elements of $\mathcal{K}(\lambda)_C^-$ have no polynomial terms.
\end{prop}
Now, consider the Dirac operator (\ref{diracopmain}) and denote its index by $\ind_\nu\slashed{\mathfrak{D}}_A^-$. Then, we have the following theorem.
\begin{thm}\cite{marshal2002deformations}
	If $\nu, \nu' \in \mathbb{R} \setminus\mathscr{D}(\slashed{\mathfrak{D}}_A^-)$ such that $\nu \leq \nu'$, then
    $$\ind_{\nu'}\slashed{\mathfrak{D}}_A^- - \ind_{\nu}\slashed{\mathfrak{D}}_A^- = \sum\{\dim \mathcal{K}(\lambda)_C^- : \lambda \in (\nu, \nu')\cap \mathscr{D}(\slashed{\mathfrak{D}}_A^-)\}.$$
\end{thm}
From the Proposition \ref{logterm}, we conclude that $\mathcal{K}(\lambda)_C^-$ is precisely the $\left(\lambda+\frac{5}{2}\right)$ eigenspace of the operator $\slashed{\mathfrak{D}}_{A_\Sigma}$. Summarising, we have the following theorem.
\begin{thm}\label{indjump}
    The Dirac operator
    $$\slashed{\mathfrak{D}}_A^- : W^{k+1,2}_{\nu-1}(\slashed{S}^-(X) \otimes \mathfrak{g}_P) \to W^{k,2}_{\nu-2}(\slashed{S}^+(X) \otimes \mathfrak{g}_P)$$
    is Fredholm if $\nu$ is not a critical weight, i.e., $\nu+\frac{5}{2} \in \mathbb{R} \setminus\Spec\slashed{\mathfrak{D}}_{A_\Sigma}$. Moreover, for two non-critical weights $\nu, \nu'$ with $\nu \leq \nu'$, the jump in the index is given by
    $$\ind_{\nu'}\slashed{\mathfrak{D}}_A^- - \ind_{\nu}\slashed{\mathfrak{D}}_A^- = \sum\limits_{\nu < \lambda < \nu'}\dim \ker\left(\slashed{\mathfrak{D}}_{A_\Sigma}-\lambda-\frac{5}{2}\right).$$
\end{thm}
\subsection{Deformations of Asymptotically Conical \texorpdfstring{$Spin(7)$}{Spin(7)}-Instantons}
Let $A$ be an asymptotically conical reference connection that also satisfies the  $Spin(7)$-instanton equation. Then, we have $\pi_7(F_A) = 0$. Now, we can write any other connection in some open neighbourhood of $A$ as $A' = A + \alpha$ for $\alpha \in \Omega^1(\mathfrak{g}_P)$. Then,
$$F_{A'} - F_A = d_A\alpha + \frac{1}{2}[\alpha, \alpha].$$
Hence the connection $A'$ is a $Spin(7)$-instanton if and only if $\pi_7\left(d_A\alpha + \frac{1}{2}[\alpha, \alpha]\right) = 0$. We also have the gauge fixing condition $d_A^*\alpha = 0$. We consider the non-linear operator
\begin{align}\label{nldirac}
    \slashed{\mathfrak{D}}_A^{\text{NL}} : \Gamma(\Lambda^1 \otimes \mathfrak{g}_P) &\to \Gamma((\Lambda^0 \oplus \Lambda^2_7)\otimes \mathfrak{g}_P)\nonumber\\
    \alpha &\mapsto \left(d_A^*\alpha, \pi_7\left(d_A\alpha + \frac{1}{2}[\alpha, \alpha]\right)\right).
\end{align}
Hence, the local moduli space of $Spin(7)$-instanton can be expressed as the zero set of $\slashed{\mathfrak{D}}_A^{\text{NL}}$. Now, from the identifications of the positive and negative spinor bundles given by $\slashed{S}^+ = \Lambda^0 \oplus \Lambda^2_7$ and $\slashed{S}^- = \Lambda^1$, and using algebraic techniques we can prove that the linearisation of the non-linear operator $\slashed{\mathfrak{D}}_A^{\text{NL}}$ is precisely the twisted linear Dirac operator $\slashed{\mathfrak{D}}_A^-$,
\begin{align}
    \slashed{\mathfrak{D}}_A^- : \Gamma(\Lambda^1 \otimes \mathfrak{g}_P) &\to \Gamma((\Lambda^0 \oplus \Lambda^2_7)\otimes \mathfrak{g}_P)\nonumber\\
    \alpha &\mapsto (d_A^*\alpha, \pi_7(d_A\alpha)).
\end{align}
In order to calculate the zero set of the non-linear operator, we calculate the kernel of the linearised Dirac operator, using the analytic techniques discussed in the previous subsections.\par
First we want to investigate the moduli space of AC $Spin(7)$-connections. We start with the following lemma which can be proved by slightly modifying the proof in \cite{karigiannis2020deformation}.
\begin{lem}
	Let $\alpha \in \Omega^{m-1,k}_\mu(\mathfrak{g}_P)$ and $\beta \in \Omega^{m,l}_\nu(\mathfrak{g}_P)$. If $k,l \geq 4$ and $\mu+\nu < -7$, then
	$$\langle d_A\alpha, \beta\rangle_{L^2} = \langle\alpha, d_A^*\beta\rangle_{L^2}.$$
\end{lem}
As an immediate consequence, we have,
\begin{cor}
	Let $f \in \Omega^{0,k+2}_\nu(\mathfrak{g}_P)$ for $\nu < 0$ and $d_A^*d_Af = 0$. Then $d_Af = 0$.
\end{cor}
\begin{proof}
	Since there are no critical weights in $(-6, 0)$, then if $d_A^*d_Af = 0$ and $\nu < 0$, we have $f \in \Omega^{0,k+2}_\mu(\mathfrak{g}_P)$ for some $\mu < -3$ for any $k$. Then $d_Af \in \Omega^{1,k+1}_{\mu-1}(\mathfrak{g}_P)$ and $\|d_Af\|_{L^2} = \langle d_A^*d_Af, f\rangle_{L^2} = 0$ by integration by parts.
\end{proof}
\begin{prop}\label{higgsvan}
	Let $f \in \Omega^{0,k+2}_\nu(\mathfrak{g}_P)$ for $\nu < 0$ and $d_A^*d_Af = 0$. Then $f = 0$. 
\end{prop}
\begin{proof}
	Since $f \in \ker(d_A^*d_A)_\nu$, we have $d_Af = 0$. Then, $d^*d|f|^2 = 2d^*\langle d_Af, f\rangle = 0$. Thus $|f|^2$ is a harmonic function. Then, using the maximum principle, it can be proved that $|f|^2 = 0$ (see \cite{marshal2002deformations}). Thus, $f = 0$.
\end{proof}\noindent
\begin{lem}\label{banclo}
	Let $X, Y$ be Banach spaces and $T : X \to Y$ be a bounded linear operator. Then $\ker T$ is closed and a closed subspace $X_0 \subset X$ is a complement of $\ker T$ if and only if $T|_{X_0}$ is injective and $T(X) = T(X_0)$.
\end{lem}
\begin{prop}\label{slicethm}
	If $\nu \in (-6, 0)$, then we have the decomposition
	$$\Omega^{1,k+1}_{\nu-1}(\mathfrak{g}_P) = \ker d_A^* \oplus \im d_A$$ 
	where $d_A : \Omega^{0,k+2}_\nu(\mathfrak{g}_P) \to \Omega^{1,k+1}_{\nu-1}(\mathfrak{g}_P)$ and $d_A^* : \Omega^{1,k+1}_{\nu-1}(\mathfrak{g}_P) \to \Omega^{0,k}_{\nu-2}(\mathfrak{g}_P)$.
\end{prop}
\begin{proof}
	Let us consider the operator $d_A^* : \Omega^{1,k+1}_{\nu-1}(\mathfrak{g}_P) \to \Omega^{0,k}_{\nu-2}(\mathfrak{g}_P)$. This is a bounded operator. Hence the kernel is a closed subspace. We want to show that $X_0 := \im d_A$ satisfies the conditions of Lemma \ref{banclo}. It can be easily proved that $\im d_A$ is a closed subspace of $\Omega^{1,k+1}_{\nu-1}(\mathfrak{g}_P)$. Then for $T:= d_A^*$ and $X := \Omega^{1,k+1}_{\nu-1}(\mathfrak{g}_P)$ we have the result.\par 
	First we show that $d_A^*$ restricted to $\im d_A$  is injective. Let, for $f, g \in \Omega^{0,k}_{\nu-2}(\mathfrak{g}_P)$, $d_A^*d_Af = d_A^*d_Ag$. Then $f-g$ is harmonic function and hence zero, which implies $d_Af - d_Ag = 0$, which establishes the injectivity.\par 
	Now, using inverse mapping theorem, it can be proved that for $\nu \in (-6, 0)$, $d_A^*d_A : \Omega^{0,k+2}_\nu(\mathfrak{g}_P) \to \Omega^{0,k+1}_{\nu-2}(\mathfrak{g}_P)$	is an isomorphism of topological vector spaces. This directly implies that $d_A^*(\Omega^{1,k+1}_{\nu-1}(\mathfrak{g}_P)) = d_A^*(\im d_A) = d_A^*d_A(\Omega^{0,k+2}_\nu(\mathfrak{g}_P))$. 
\end{proof}
Let us consider the moduli space of connections $\mathcal{B}_{k+1,\nu} = \mathcal{A}_{k+1,\nu-1}/\mathcal{G}_{k+2,\nu}$. Then the infinitesimal action of the gauge group $\mathcal{G}_{k+2,\nu}$ is given by
$$-d_A : \Omega^{0,k+2}_\nu(\mathfrak{g}_P) \to \Omega^{1,k+1}_{\nu-1}(\mathfrak{g}_P).$$
Thus we can view the Proposition \ref{slicethm} as a ``\define{slice theorem}'' which gives us the complement of the action of the gauge group.
\begin{lem}\label{freeact}
	The action of the gauge group $\mathcal{G}_{k+2,\nu}$ on the space of connections $\mathcal{A}_{k+1,\nu-1}$ is free.
\end{lem}
\begin{proof}
	Let us consider the stabilizer group $\Gamma_{A, \nu} = \{\varphi \in \mathcal{G}_{k+2,\nu} : \varphi \cdot A = A\}$. We consider gauge transformations as sections of $\End(V)$. Now, the connection $A$ has holonomy contained in $G$ and hence it preserves the inner product on $V$ as well as on $\End(V)$. Since, $\varphi \in\Gamma_{A, \nu}$ is a gauge transformation, by definition \ref{gaugetrans}, we have $\|\varphi - I\| \in \Omega^{0,k+2}_\nu(\mathfrak{g}_P)$ and hence, $$d^*d\|\varphi - I\|^2 = 2d^*\langle d_A(\varphi - I), \varphi - I\rangle = 0 \Rightarrow \varphi - I = 0,$$ since, $\varphi \cdot A = A$ implies $d_A\varphi = 0$. Hence, we have $\Gamma_{A, \nu} = \{I\}$.
\end{proof}\noindent
We note that unlike the case where the manifold $X$ is compact, when $X$ is AC, reducible connections do not produce singularities in the space of connections modulo gauge.\par 
Let us define the set 
$$T_{A, \nu,\epsilon} := \left\{\alpha \in \Omega^{1,k+1}_{\nu-1}(\mathfrak{g}_P) : d_A^*\alpha = 0, \|\alpha\|_{W^{k+1, 2}_{\nu-1}} < \epsilon\right\}.$$
Then $T_{A, \nu,\epsilon} \subset \ker d_A^*$ models a local neighbourhood of the moduli space $\mathcal{B}_{k+1,\nu}$.
We note that studying the moduli space using the local model $T_{A, \nu,\epsilon}$ is basically same as solving the Coulomb gauge fixing condition $d_A^*\alpha = 0$. This condition is local: locally, near $A$ it selects a unique gauge equivalent class. Due to the topological obstructions, a global gauge fixing condition fails to exist. The following lemma provides a sufficient condition for solving the gauge fixing condition. It is the weighted version of Proposition 2.3.4 of \cite{donaldson1990geometry}.
\begin{lem}\cite{driscoll2020paper}\label{surjformod}
	If $\nu \in (-6, 0)$, then there is a constant $c(A) > 0$ such that if $A' \in \mathcal{A}_{\nu-1}$ and $A' = A + \alpha$ satisfies $\|\alpha\|_{W^{4,2}_{\nu-1}} < c(A)$ then there is a gauge transformation $ \varphi \in \mathcal{G}_\nu$ such that $\varphi(A')$ is in Coulomb gauge relative to $A$.
\end{lem}
\begin{prop}
	If $\nu \in (-6, 0)$, then the moduli space $\mathcal{B}_{k+1,\nu}$ is a smooth manifold and the sets $T_{A, \nu,\epsilon}$ provide charts near $[A] \in \mathcal{B}_{k+1,\nu} = \mathcal{A}_{k+1,\nu-1}/\mathcal{G}_{k+2,\nu}$.
\end{prop}
\begin{proof}
    The smoothness follows from Lemma \ref{freeact} and the surjectivity follows from Proposition \ref{surjformod}. The homeomorphism between $T_{A, \nu,\epsilon}$ and a neighbourhood of $[A] \in \mathcal{B}_{k+1,\nu}$ follows from a weighted version of Proposition 4.2.9 of \cite{donaldson1990geometry} and the fact that $\Gamma_{A, \nu} = \{I\}$.
\end{proof}
Now, we turn out focus to the main objective of this section: the moduli space of AC $Spin(7)$-instantons. Let us define the spaces,
$$\mathcal{M}(A_\Sigma, \nu)_{k+1} := \{A \in \mathcal{A}_{k+1, \nu-1} : A \text{ is a $Spin(7)$-instanton on $P$}\}/\mathcal{G}_{k+2, \nu}.$$ 
The proof of the following proposition is a weighted version of Proposition 4.2.16 of \cite{donaldson1990geometry} and very similar to the proof for $7$-dimensional case given by Driscoll \cite{driscoll2020paper}.
\begin{prop}\label{reg}
	If $k \geq 4$ and $\nu \in (-6, 0)$, then the natural inclusion given by $\mathcal{M}(A_\Sigma, \nu)_{k+1} \hookrightarrow \mathcal{M}(A_\Sigma, \nu)_k$ is a homeomorphism.
\end{prop}
Hence by Proposition \ref{reg} and weighted Sobolev embedding theorem, we see that $\mathcal{M}(A_\Sigma, \nu)_k$ consists of smooth connections.
We obtain the following important corollary.
\begin{cor}\label{regucor}
	If $\nu \in (-6, 0)$, then the zero set of the non-linear twisted Dirac operator  
	$$\left(\slashed{\mathfrak{D}}^{\text{NL}}_A\right)^- : W^{k+1,2}_{\nu-1}(\slashed{S}^-(X) \otimes \mathfrak{g}_P) \to W^{k,2}_{\nu-2}(\slashed{S}^+(X) \otimes \mathfrak{g}_P)$$
	is independent of $k \geq 4$. Moreover, a neighbourhood of $[A] \in \mathcal{M}(A_\Sigma, \nu)$ is homeomorphic to $0$ in $\left(\slashed{\mathfrak{D}}^{\text{NL}}_A\right)^{-1}(0)$.
\end{cor}
Finally, we have all the tools necessary to define the deformation and obstruction spaces, state and prove the main theorem of this section.
\begin{defin}
	For $\nu < 0$ the \define{space of infinitesimal deformations} is defined to be 
	\begin{equation}\label{defsp}
	       \mathcal{I}(A, \nu) := \left\{\alpha \in \Omega^{1, k+1}_{\nu-1}(\mathfrak{g}_P) : \slashed{\mathfrak{D}}_A^-\alpha = 0\right\}.
	\end{equation}
	The \define{obstruction space} $\mathcal{O}(A, \nu)$ is a finite-dimensional subspace of $\Omega^{0, k}_{\nu-2}(\mathfrak{g}_P) \oplus \Omega^{2, k}_{\nu-2}(\mathfrak{g}_P)$ such that,
    \begin{equation}\label{obssp}
        \Omega^{0, k}_{\nu-2}(\mathfrak{g}_P) \oplus \Omega^{2, k}_{\nu-2}(\mathfrak{g}_P) = \slashed{\mathfrak{D}}_A^-\left(\Omega^{1, k+1}_{\nu-1}(\mathfrak{g}_P)\right) \oplus \mathcal{O}(A, \nu).
    \end{equation}
\end{defin}\noindent
We note that $\mathcal{I}(A, \nu)$ and $\mathcal{O}(A, \nu)$ are precisely the kernel and cokernel of the twisted Dirac operator corresponding to the rate $\nu$. We have the main theorem:
\begin{thm}\label{mainthm}
	Let $A$ be an AC $Spin(7)$-instanton asymptotic to a nearly $G_2$ instanton $A_\Sigma$. Moreover, let $\nu \in (\mathbb{R} \setminus \mathscr{D}(\slashed{\mathfrak{D}}_A^-)) \cap (-6,0)$. Then there exists an open neighbourhood $\mathcal{U}(A, \nu)$ of $0$ in $\mathcal{I}(A, \nu)$, and a smooth map $\kappa : \mathcal{U}(A, \nu) \to \mathcal{O}(A, \nu)$, with $\kappa(0) = 0$, such that an open neighbourhood of $0 \in \kappa^{-1}(0)$ is homeomorphic to a neighbourhood of $A$ in $\mathcal{M}(A_\Sigma, \nu)$. Hence, the virtual dimension of the moduli space is given by $\dim\mathcal{I}(A, \nu) - \dim\mathcal{O}(A, \nu)$. Moreover, $\mathcal{M}(A_\Sigma, \nu)$ is a smooth manifold if $\mathcal{O}(A, \nu) = \{0\}$.
\end{thm}
\begin{proof}
	Let $\mathbf{X} = \Omega^{1, k+1}_{\nu-1}(\mathfrak{g}_P) \times \mathcal{O}(A, \nu)$ and $\mathbf{Y} = \Omega^{0, k}_{\nu-2}(\mathfrak{g}_P) \oplus \Omega^{2, k}_{\nu-2}(\mathfrak{g}_P)$. We define a Banach space morphism as
	\begin{align*}
		\mathbf{F} : \mathbf{X} &\to \mathbf{Y}\\
		(\alpha, \beta) &\mapsto \slashed{\mathfrak{D}}^{\text{NL}}_A\alpha+\beta.
	\end{align*}
	where $\slashed{\mathfrak{D}}^{\text{NL}}_A$ is the nonlinear twisted Dirac operator (\ref{nldirac}). Then, $\mathbf{F}(0, 0) = 0$, and the differential at $(0, 0)$ is given by
    \begin{align*}
	   d\mathbf{F}|_{(0, 0)} : \mathbf{X} &\to \mathbf{Y}\\
	   (\alpha, \beta) &\mapsto \slashed{\mathfrak{D}}_A\alpha+\beta.
    \end{align*}
    Then $d\mathbf{F}|_{(0, 0)}$ is surjective and $d\mathbf{F}|_{(0, 0)} = 0$ if and only if $(\slashed{\mathfrak{D}}_A\alpha, \beta) = (0,0)$. Hence $\ker d\mathbf{F}|_{(0, 0)} =: \mathbf{K} = \mathcal{I}(A, \nu) \times \{0\}$ is finite dimensional, and we have a decomposition of $\mathbf{X}$ as $\mathbf{X} = \mathbf{K} \oplus \mathbf{Z}$, where $\mathbf{Z}\subset \mathbf{X}$ is a closed subspace. Moreover, we can write $\mathbf{Z} = \mathcal{Z} \times \mathcal{O}(A, \nu)$ for a closed subset $\mathcal{Z} \subset \Omega^{1, k+1}_{\nu-1}(\mathfrak{g}_P)$. By implicit function theorem, we choose the open subsets $\mathcal{U} \subset \mathcal{I}(A, \nu),\ \mathcal{V}_1 \subset \mathcal{Z}$ and $\mathcal{V}_2 \subset \mathcal{O}(A, \nu)$, and smooth maps $\mathcal{F}_i : \mathcal{U} \to \mathcal{V}_i$ for $i = 1, 2$, such that
    $$\mathbf{F}^{-1}(0) \cap ((\mathcal{U} \times \mathcal{V}_1)\times \mathcal{V}_2) = \{((\alpha, \mathcal{F}_1(\alpha)), \mathcal{F}_2(\alpha)) : \alpha \in \mathcal{U}\}$$
    in $\mathbf{X} = (\mathcal{I}(A, \nu) \oplus \mathcal{Z}) \times \mathcal{O}(A, \nu)$. Hence the kernel of $\mathbf{F}$ near $(0, 0)$ is diffeomorphic to an open subset of $\mathcal{I}(A, \nu)$ containing $0$.\par 
    Now, we define $\mathcal{U}(A, \nu) := \mathcal{U}$ and the map
    \begin{align*}
	   \kappa : \mathcal{U}(A, \nu) &\to \mathcal{O}(A, \nu)\\
	   \alpha &\mapsto \mathcal{F}_2(\alpha).
    \end{align*}
    Then we have a homeomorphism from an open neighbourhood of $0$ in $\kappa^{-1}(0)$ to an open neighbourhood of $0$ in $\left(\slashed{\mathfrak{D}}^{\text{NL}}_A\right)^{-1}(0)$ given by $\alpha \mapsto (\alpha, \mathcal{F}_1(\alpha))$. Now, corollary \ref{regucor} tells us that a neighbourhood of $[A] \in \mathcal{M}(A, \nu)$ is homeomorphic to a neighbourhood of $0$ in $\left(\slashed{\mathfrak{D}}^{\text{NL}}_A\right)^{-1}(0)$. Hence the theorem.
\end{proof}

\section{Eigenvalues of the Twisted Dirac Operator on \texorpdfstring{$S^7$}{S7}}\label{section4}
In order to study the deformations of FNFN instantons, we need to calculate the spectrum of the twisted Dirac operator on the link $S^7$. We use representation theory, Frobenius reciprocity, and Casimir operators, to write the Dirac operators as a sum of Casimir operators. Then the problem of finding the spectrum of Dirac operator reduces to finding the eigenvalues of the Casimir operators. This method relies on $S^7$ being a homogeneous manifold and is developed based on the works of \cite{strese80}, \cite{bar1991dirac}, \cite{bar92}, \cite{driscoll2020thesis}.
\subsection{Dirac operators on Homogeneous Nearly \texorpdfstring{$G_2$}{G2}-Manifolds}
Let $\Sigma = G/H$ be a reductive homogeneous nearly $G_2$-manifold. We consider the principal $H$-bundle $G \to \Sigma$. Let $\rho_V : H \to \Aut(V)$ be a representation of $H$. Then we have the associated vector bundle $E := G \times_\rho V$ and the space of smooth sections $\Gamma(E)$ can be identified with the space of $H$-equivariant smooth function $G \to V$, i.e. the space $C^\infty(G, V)^H$.\par 
Now, the left action of $G$ on the space $L^2(G, V)^H$ gives a representation $\rho_L$, called the \define{left regular representation} defined by 
$$\rho_L(h)\eta(g) = \eta(h^{-1}g),$$ 
for $\eta \in L^2(G, V)^H$, and $g, h \in G$.\par 
The right action of $G$ on the space $L^2(G, V)$ gives a representation $\rho_R$, called the \define{right regular representation} defined by 
$$\rho_R(h)\eta(g) = \eta(gh).$$
Then from $H$-equivariance,
$$L^2(G, V)^H = \{\eta \in L^2(G, V) : \rho_R(h)\eta = \rho_V(h)^{-1}\eta \text{ for all }h \in H\}.$$
Now, since $G/H$ is reductive, we have an orthogonal decomposition $\mathfrak{g} = \mathfrak{h}\oplus\mathfrak{m}$ induced by the Killing form $K$ on $G$. We define a nearly $G_2$-metric given by \cite{singhal2021paper}
\begin{equation}\label{killmet}
    g(X, Y) = -\frac{3}{40}K(X, Y)
\end{equation}
Let $\{I_A\}$ be an orthonormal basis for $\mathfrak{g}$, $\{I_a : a = 1, \dots, \dim(G/H) = 7\}$ is a orthonormal basis for $\mathfrak{m}$ and $\{I_i : i = \dim(H)+1, \dots, \dim(G)\}$ is a orthonormal basis of $\mathfrak{h}$.\par 
We note that in this framework $G$-invariant tensors on the tangent bundle  $T(G/H)$ correspond to $H$-invariant tensors on $\mathfrak{m}$ \cite{KobayashiNomizu}.\par  
Now, we consider the complex spinor bundle $\slashed{S}(\Sigma) = G \times_{\rho, H}\Delta$ where $\Delta$ is the spinor space. From the splitting $\slashed{S}_\mathbb{C}(\Sigma) \cong \Lambda^0_\mathbb{C}\oplus\Lambda^1_\mathbb{C}$, we have $\Delta \cong \mathbb{C}\oplus\mathfrak{m}_\mathbb{C}^*$. We now twist the spinor bundle by the associated bundle $E = G \times_{\rho_V, H}V$ for a representation $V$ of $H$. Then 
$$\slashed{S}_\mathbb{C}(\Sigma)\otimes E = G\times_{(\rho_V\otimes \rho_\Delta, H)}\Delta \otimes V.$$
The canonical connection $\nabla^{1,A_\Sigma} : L^2(G, \Delta \otimes V)^H \to L^2(G, \mathfrak{m}^* \otimes \Delta \otimes V)^H$ can be written as 
$$\nabla^{1,A_\Sigma}\eta = e^a \otimes \rho_R(I_a)\eta$$
where $e^a$ is the basis of $\mathfrak{m}^*$ dual to $I_a$ and $\eta \in L^2(G, \Delta \otimes V)^H$. Then the Dirac operator $\slashed{\mathfrak{D}}^1_{A_\Sigma}$ is given by
\begin{equation}\label{candirac}
    \slashed{\mathfrak{D}}^1_{A_\Sigma} = I_a \cdot\rho_R(I_a).
\end{equation}
Then, from, (\ref{candirac}), we have a family of Dirac operators 
\begin{equation}\label{diracfamily}
    \slashed{\mathfrak{D}}^t_{A_\Sigma} = \slashed{\mathfrak{D}}^1_{A_\Sigma}+\frac{(t-1)}{2}\phi
\end{equation}
where for $t = 0$, we have $\slashed{\mathfrak{D}}^0_{A_\Sigma} = \slashed{\mathfrak{D}}_{A_\Sigma}$.\par 
Let $\widehat{G}$ be the set of equivalence classes of irreducible representations of $G$ and for $\gamma \in \widehat{G}$ we have a representative $(V_\gamma, \rho_\gamma)$. Then \define{Frobenius reciprocity} implies the decomposition 
\begin{equation}\label{frob2}
    L^2(\slashed{S}_\mathbb{C}(\Sigma)\otimes E) \cong L^2(G, \Delta \otimes V)^H \cong \bigoplus\limits_{\gamma\in\widehat{G}}\Hom(V_\gamma, \Delta \otimes V)^H\otimes V_\gamma
\end{equation}
where the action of $G$ on $V_\gamma$ of the right hand side of the expression corresponds to the action of $\rho_L$ on $L^2(\slashed{S}_\mathbb{C}(\Sigma)\otimes E)$. We note that the Dirac operator commutes with the left action of $G$ and hence it respects the decomposition (\ref{frob2}). Then, by Schur's lemma, for every $t \in \mathbb{R}$, the Dirac operator $\slashed{\mathfrak{D}}^t_{A_\Sigma}$, restricted to $\Hom(V_\gamma, \Delta \otimes V)^H\otimes V_\gamma$ is given by
\begin{equation}
    \slashed{\mathfrak{D}}^t_{A_\Sigma}|_{\Hom(V_\gamma, \Delta \otimes V)^H\otimes V_\gamma} = \left(\slashed{\mathfrak{D}}^t_{A_\Sigma}\right)_\gamma \otimes \Id,
\end{equation}
where $\left(\slashed{\mathfrak{D}}^t_{A_\Sigma}\right)_\gamma :\Hom(V_\gamma, \Delta \otimes V)^H\to\Hom(V_\gamma, \Delta \otimes V)^H$ is the Dirac operator \cite{driscoll2020thesis}
\begin{equation}
    \left(\slashed{\mathfrak{D}}^t_{A_\Sigma}\right)_\gamma\eta = -I_a \cdot (\eta \circ \rho_{V_\gamma}(I_a))+\frac{t-1}{2}\phi\cdot\eta.
\end{equation}\par
If $\{I_A\}$ is an orthonormal basis of $\mathfrak{g}$, then the \define{Casimir Operator} $\Cas_\mathfrak{g} \in \sym^2(\mathfrak{g})$ is the inverse of the metric on $\mathfrak{g}$, defined by
$$\Cas_\mathfrak{g} = \sum\limits_{A=1}^{\dim{G}}I_A \otimes I_A.$$
If $(\rho, V)$ is any representation of $\mathfrak{g}$, then $\rho(\Cas_\mathfrak{g}) = \sum\limits_{A=1}^{\dim{G}}\rho(I_A)^2$.\par 
Using Lichnerowicz formula, we can write the square of the Dirac operator as sum of Casimir operators.
\begin{prop}\cite{singhal2021paper}\label{lichnerog2}
	Let $V$ be a representation of $H$, $E$ be the associated vector bundle $G \to G/H$, ans $A_{\Sigma}$ be the canonical connection of $E$. Then,
	\begin{equation}
		\left(\slashed{\mathfrak{D}}^{1/3}_{A_\Sigma}\right)^2\eta = (-\rho_L(\Cas_\mathfrak{g})+\rho_V(\Cas_\mathfrak{h})+49/9)\eta
	\end{equation}
for $\eta \in \Gamma(\slashed{S}_\mathbb{C}(\Sigma)\otimes E)$.
\end{prop}\noindent
Restricting the operator $\left(\slashed{\mathfrak{D}}^{1/3}_{A_\Sigma}\right)^2$ to $\Hom(V_\gamma, \Delta \otimes V)^H\otimes V_\gamma$, we get
\begin{equation}
    \left(\slashed{\mathfrak{D}}^{1/3}_{A_\Sigma}\right)^2|_{\Hom(V_\gamma, \Delta \otimes V)^H\otimes V_\gamma} = \left(\slashed{\mathfrak{D}}^{1/3}_{A_\Sigma}\right)_\gamma^2 \otimes \Id.
\end{equation}
The self-adjointness of this operator implies that it is diagonalisable with real eigenvalues.  Frobenius reciprocity and Proposition \ref{lichnerog2} implies
\begin{equation}\label{lichnerog3}
		\left(\slashed{\mathfrak{D}}^{1/3}_{A_\Sigma}\right)^2_\gamma = -\rho_{V_\gamma}(\Cas_\mathfrak{g})+\rho_V(\Cas_\mathfrak{h})+49/9.
\end{equation}

\subsubsection*{Eigenvalue Bounds}
We have the nearly $G_2$-manifold $G/H$ which is a reductive homogeneous space. Now, the Casimir operators commute with the group action, and hence on irreducible representations, they act as a multiple of the identity. That is,
$$\rho_\gamma(\Cas_{\mathfrak{g}}) = c^\mathfrak{g}_\gamma\Id,\ \ \rho_\gamma(\Cas_{\mathfrak{h}}) = c^\mathfrak{h}_\gamma\Id,$$
where $c^\mathfrak{g}_\gamma$ and $c^\mathfrak{h}_\gamma$ are real numbers, called \define{Casimir eigenvalues}.
Now, let $V_\gamma$ be an irreducible representation of $G$. Then we have the decomposition of $V_\gamma$ as
$$V_\gamma = \bigoplus_{\sigma\in I}W_\sigma^\gamma$$ 
where $W_\sigma$ are irreducible representations of $H$ and $I$ is a finite sequence in $\widehat{H}$ which may have repeated entries. Similarly, for finite sequences $J, K$ in $\widehat{H}$, we have the decomposition 
$$\Delta = \bigoplus_{\alpha \in J}W_\alpha,\ \ V = \bigoplus_{\beta \in K}W_\beta.$$
Let us assume that in the decomposition of $W_\alpha \otimes W_\beta$ in to irreducible representations, $W^\gamma_\sigma$ occurs with multiplicity $1$. Then we consider the composition map
$$q^\sigma_{\alpha\beta} : V_\gamma \to W^\gamma_\sigma \to W_\alpha \otimes W_\beta \hookrightarrow \Delta \otimes V$$
where the first map is projection map and the third one is equivariant embedding. Since the decomposition of $V_\gamma$, $\Delta$ and $V$ into irreducible representations of $H$ are orthogonal, $\{q^\sigma_{\alpha\beta}\}$ is an orthogonal basis of $\Hom(V_\gamma, \Delta\otimes V)^H$. Hence, $q^\sigma_{\alpha\beta}$ are eigenvectors of (\ref{lichnerog3}) and
\begin{align}\label{diag1/3}
    \left(\slashed{\mathfrak{D}}^{1/3}_{A_\Sigma}\right)^2_\gamma(q^\sigma_{\alpha\beta}) = \left(-c^\mathfrak{g}_\gamma+c^\mathfrak{h}_\beta+49/9\right)q^\sigma_{\alpha\beta}.
\end{align}
Then $\{q^\sigma_{\alpha\beta}\}$ diagonalizes the twisted Dirac operator $\left(\slashed{\mathfrak{D}}^{1/3}_{A_\Sigma}\right)^2_\gamma$. The eigenvalues are given by $-c^\mathfrak{g}_\gamma+c^\mathfrak{h}_\beta+49/9$ with multiplicities $\dim\Hom(V_\gamma, \Delta\otimes W_\beta)^H$. Hence the eigenvalues of $\left(\slashed{\mathfrak{D}}^{1/3}_{A_\Sigma}\right)_\gamma$ are $\sqrt{-c^\mathfrak{g}_\gamma+c^\mathfrak{h}_\beta+49/9}$ and $-\sqrt{-c^\mathfrak{g}_\gamma+c^\mathfrak{h}_\beta+49/9}$.\par 
Now, we want to find an eigenvalue bound for the operator $\left(\slashed{\mathfrak{D}}^0_{A_\Sigma}\right)_\gamma$.
\begin{thm}\label{eigbnd}
    Let $V_\gamma$ be an irreducible representation of $G$. If 
    $$L_\gamma := \sqrt{\min_\beta\left\{-c^\mathfrak{g}_\gamma+c^\mathfrak{h}_\beta+49/9\right\}}-\frac{7}{6} > 0,$$
    then $L_\gamma$ is a lower bound on the smallest positive eigenvalues of $\left(\slashed{\mathfrak{D}}^0_{A_\Sigma}\right)_\gamma$.
\end{thm}
\begin{proof}
    We have $\displaystyle\left(\slashed{\mathfrak{D}}^0_{A_\Sigma}\right)_\gamma = \left(\slashed{\mathfrak{D}}^{1/3}_{A_\Sigma}\right)_\gamma - \frac{1}{6}\phi$. Now, $\phi$ acts on $\Lambda^0$ and $\Lambda^1$ with eigenvalues $7$ and $-1$ respectively. If $\lambda_1^2, \dots, \lambda_n^2$ are eigenvalues for $\left(\slashed{\mathfrak{D}}^{1/3}_{A_\Sigma}\right)_\gamma^2$, then $\left(\slashed{\mathfrak{D}}^{1/3}_{A_\Sigma}\right)_\gamma$ has eigenvalues $\pm\lambda_1, \dots, \pm\lambda_n$. Then, the the smallest positive eigenvalue of $\left(\slashed{\mathfrak{D}}^0_{A_\Sigma}\right)_\gamma$ is bounded below by difference of the smallest positive eigenvalue of $\left(\slashed{\mathfrak{D}}^{1/3}_{A_\Sigma}\right)_\gamma$ and $\frac{1}{6}\max\{|7|, |-1|\}$, i.e., by $\sqrt{\min\limits_\beta\left\{-c^\mathfrak{g}_\gamma+c^\mathfrak{h}_\beta+49/9\right\}}-\frac{7}{6}$.
\end{proof}
\subsection{The Twisted Dirac Operator on \texorpdfstring{$S^7$}{S7}}
We identify $S^7$ with the homogeneous space $Spin(7)/G_2$. Let $\mathfrak{m}$ be the orthogonal complement of $\mathfrak{g}_2 \subset \mathfrak{spin}(7)$ with respect to the Killing form on the Lie algebra $\mathfrak{spin}(7)$. Clearly, $[\mathfrak{g}_2, \mathfrak{m}] \subset \mathfrak{m}$ and hence the homogeneous space $Spin(7)/G_2$ is reductive. Consider the Maurer--Cartan form $\theta$ on $Spin(7)$ and the splitting $\theta = \theta_{\mathfrak{g}_2} \oplus \theta_\mathfrak{m}$ induced by the decomposition $\mathfrak{spin}(7) = \mathfrak{g}_2 \oplus \mathfrak{m}$. Then $\theta_{\mathfrak{g}_2} =: A_\Sigma$ is canonical connection on the bundle $G_2 \to Spin(7) \to S^7$. In (\ref{torsion}) putting $t = 1$ for canonical connection and denoting $T^1(X, Y)$ by $T(X, Y)$, the torsion is given by
\begin{equation}\label{fandphi}
	T(X, Y) = -[X, Y]_{\mathfrak{m}} = \frac{2}{3}\phi(X, Y, \cdot).
\end{equation}
Let $\{I_A : A = 1, \dots, 21\}$ be an orthonormal basis for $\mathfrak{spin}(7)$, $\{I_a : a = 1, \dots, 7\}$ be a basis for $\mathfrak{m}$ and $\{I_i : i = 8, \dots, 21\}$ be a basis for $\mathfrak{g}_2$. Let us consider $Cl(7)$, the Clifford algebra over $\mathbb{R}^7$. Let $\Delta$ be a $8$-dimensional representation and $\rho_\Delta : \mathfrak{spin}(7) \to \End(\Delta)$ be the restriction to $\mathfrak{spin}(7)$.
From (\ref{frob2}), recall the identification $\Gamma(\slashed{S}_\mathbb{C}(\Sigma) \otimes (\mathfrak{g}_P)_\mathbb{C}) \cong L^2(Spin(7), \Delta \otimes V)^{G_2}$, where we now choose $V \cong \mathfrak{spin}(7)_\mathbb{C}$. Consider the operator
\begin{align}\label{diracrho1}
	&\slashed{\mathfrak{D}}_{A_\Sigma}^\rho : L^2(Spin(7), \Delta \otimes V)^{G_2} \to L^2(Spin(7), \Delta \otimes V)^{G_2}\nonumber\\
	&\slashed{\mathfrak{D}}_{A_\Sigma}^\rho = \rho_\Delta(I_a)\rho_R(I_a).
\end{align}
We want to compare the operator (\ref{diracrho1}) with the Dirac operator (\ref{candirac}). That is, compare $\rho_\Delta(I_a)$ with the Clifford multiplication by $I_a$. Let $\{e^a : a = 1, \dots, 7\}$ be a orthonormal basis of $\mathfrak{m}^*$ dual to $I_a$ of $\mathfrak{m}$. Now, we have the decomposition $\mathfrak{spin}(7) = \mathfrak{g}_2 \oplus \mathfrak{m}$. We identify $\mathfrak{m}$ with $\mathbb{R}^7$ via an isomorphism $F : \mathfrak{m} \to \mathbb{R}^7$ as representations of $\mathfrak{g}_2$. Then we note that $F$ is unique by Schur's lemma and imposing the condition that $F$ commutes with the $G_2$-structures, i.e., $F$ maps the $G_2$-invariant $3$-form of $\mathfrak{m}$ to $G_2$-invariant $3$-form of $\mathbb{R}^7$. Then $F$ induces an isomorphism $Cl(\mathfrak{m}) \cong Cl(7)$. Then from (\ref{spinor-iso}) we have an isomorphism $\Delta \cong \mathbb{C} \oplus \mathfrak{m}_\mathbb{C}^*$. The basis $e^a$ for $\mathfrak{m}^*$ gives orthonormal vectors $e^a$ in $\Delta$. We choose $e^0 \in \Delta$ so that $\{e^0, e^1, \dots,e^7\}$ is an orthonormal basis for $\Delta$ and the $Spin(7)$-invariant $4$-form $\Phi$ is given by the usual formula $e^0\wedge\phi+*\phi$. We identify $(f, v) \in \mathbb{C} \oplus \mathfrak{m}_\mathbb{C}^*$ with $f e^0 + v \in \Lambda^1(\Delta)^*$.\par
\begin{lem}
The action of $\rho_\Delta(I_a)$ on $\Delta$ is given by
    \begin{align}\label{rhodelta}
	\rho_\Delta(I_a)(f, v) = \left(-\langle e^a, v\rangle, fe^a - \frac{1}{3}((e^a \wedge v) \intprod \phi)\right).
\end{align}
\end{lem}
\begin{proof}
    It can be easily shown that $\Phi$ is invariant under the action of $\rho_\Delta(I_a)$. Schur's lemma implies that the action should be a constant multiple of the right hand side of the equation \ref{rhodelta}. But from (\ref{fandphi}) we also have $[\rho_\Delta(I_a),\rho_\Delta(I_b)]_{\mathfrak{m}}=-\frac{2}{3}\phi_{abc}\rho_{\Delta}(I_c)$, which implies the constant is $1$.
\end{proof}
Now, the formula (\ref{rhodelta}) does not agree with the formula for Clifford multiplication by $I_a$, which from (\ref{cliffmult1}), is given by 
\begin{equation}\label{cliffmultea}
    I_a \cdot (f, v) = \left(\langle e^a, v\rangle, -fe^a - (e^a \wedge v) \intprod \phi\right).
\end{equation}
To fix this, we consider another representation of $\mathfrak{spin}(7)$ defined by 
$$\widetilde{\rho}_\Delta(X) := M^{-1}\cdot\rho_\Delta(X)\cdot M$$
where $M := \begin{pmatrix}
	1&0\\
	0 &-\Id_\mathfrak{m}
\end{pmatrix} = M^{-1}$. Then we calculate
\begin{equation}\label{rhodeltatilda}
    \widetilde{\rho}_\Delta(I_a)(f, v) = \left(\langle e^a, v\rangle, -fe^a - \frac{1}{3}((e^a \wedge v) \intprod \phi)\right).
\end{equation}
Thus from (\ref{cliffmultea}), (\ref{rhodelta}) and (\ref{rhodeltatilda}), the Clifford multiplication of a spinor $\eta$ by $I_a$ can be rewritten as
\begin{equation}\label{cliffmultfordirac}
    I_a \cdot \eta = \left(\rho_\Delta(I_a)+2\widetilde{\rho}_\Delta(I_a)\right)\eta.
\end{equation}
Consider the operators $\slashed{\mathfrak{D}}_{A_\Sigma}^\rho$ given by (\ref{diracrho1}) and
\begin{align}\label{diracrho2}
    &\widetilde{\slashed{\mathfrak{D}}}_{A_\Sigma}^\rho : L^2(Spin(7), \Delta \otimes V)^{G_2} \to L^2(Spin(7), \Delta \otimes V)^{G_2}\nonumber\\
	&\widetilde{\slashed{\mathfrak{D}}}_{A_\Sigma}^\rho = \widetilde{\rho}_\Delta(I_a)\rho_R(I_a).
\end{align}
Note that the operators $\slashed{\mathfrak{D}}_{A_\Sigma}^\rho$ and $\widetilde{\slashed{\mathfrak{D}}}_{A_\Sigma}^\rho$ commute with the left action of $Spin(7)$ and hence respect the decomposition (\ref{frob2}). From (\ref{cliffmultfordirac}), we have
\begin{align}\label{dirac1}
	\left(\slashed{\mathfrak{D}}_{A_\Sigma}^\rho+2\widetilde{\slashed{\mathfrak{D}}}_{A_\Sigma}^\rho\right) = \rho_\Delta(I_a)\rho_R(I_a)+2\widetilde{\rho}_\Delta(I_a)\rho_R(I_a)
	= I_a \cdot\rho_R(I_a) = \slashed{\mathfrak{D}}_{A_\Sigma}^1.
\end{align}
Hence from (\ref{diracfamily}) and (\ref{dirac1}), we have
\begin{equation}\label{diract}
    \slashed{\mathfrak{D}}_{A_\Sigma}^t = \slashed{\mathfrak{D}}_{A_\Sigma}^1 + \frac{t-1}{2}\phi = \left(\slashed{\mathfrak{D}}_{A_\Sigma}^\rho+2\widetilde{\slashed{\mathfrak{D}}}_{A_\Sigma}^\rho\right) + \frac{t-1}{2}\phi.
\end{equation}
In terms of Casimir operators, we can write
\begin{align*}
    \slashed{\mathfrak{D}}_{A_\Sigma}^\rho = \frac{1}{2}\left(\rho_{\Delta\otimes R}(\Cas_\mathfrak{m}) - \rho_\Delta(\Cas_\mathfrak{m}) - \rho_R(\Cas_\mathfrak{m})\right)\\
    \widetilde{\slashed{\mathfrak{D}}}_{A_\Sigma}^\rho = \frac{1}{2}\left(\widetilde{\rho}_{\Delta\otimes R}(\Cas_\mathfrak{m}) - \widetilde{\rho}_\Delta(\Cas_\mathfrak{m}) - \rho_R(\Cas_\mathfrak{m})\right)
\end{align*}
where $\rho(\Cas_\mathfrak{m}) := \rho(\Cas_{\mathfrak{spin}(7)})-\rho(\Cas_{\mathfrak{g}_2}) = \rho(I_a)\rho(I_a)$. Then restricting $\slashed{\mathfrak{D}}_{A_\Sigma}^\rho$ and $\widetilde{\slashed{\mathfrak{D}}}_{A_\Sigma}^\rho$ to $\Hom(V_\gamma, \Delta \otimes \mathfrak{spin}(7)_\mathbb{C})^{G_2}$, we get
\begin{align}
    \left(\slashed{\mathfrak{D}}_{A_{\Sigma}}^\rho\right)_\gamma &= \frac{1}{2}\left(\rho_{\Delta \otimes V_\gamma^*}(\Cas_{\mathfrak{m}})-\rho_\Delta(\Cas_{\mathfrak{m}})-\rho_{V_\gamma^*}(\Cas_{\mathfrak{m}})\right)\label{diracrho}\\
    \left(\widetilde{\slashed{\mathfrak{D}}}_{A_\Sigma}^\rho\right)_\gamma &= M_\gamma^{-1}\left(\slashed{\mathfrak{D}}_{A_{\Sigma}}^\rho\right)_\gamma M_\gamma\label{diracrhotil}
\end{align}
where we define $M_\gamma$ in the following way. Recalling $M = \begin{pmatrix}
1 &0\\
0 &-\Id_{\mathfrak{m}} 
\end{pmatrix}: \Delta \to \Delta$, for $\xi \in \Hom(V_\gamma \otimes \Delta, \mathfrak{spin}(7)_\mathbb{C})^{G_2}$, we define $M_\gamma$ by
\begin{equation}\label{m-gamma}
    M_\gamma\xi(v \otimes \delta) := \xi(v \otimes M\delta).
\end{equation}
Then from (\ref{diract}) we have
\begin{equation}\label{dirac-t-gamma}
    \left(\slashed{\mathfrak{D}}_{A_\Sigma}^t\right)_\gamma = \left(\slashed{\mathfrak{D}}_{A_\Sigma}^\rho\right)_\gamma+2\left(\widetilde{\slashed{\mathfrak{D}}}_{A_\Sigma}^\rho\right)_\gamma + \frac{t-1}{2}\phi.
\end{equation}\par
Now, we want to calculate the eigenvalues of the Casimir operators appearing in the expression of the Dirac operator. Let $V_{(a,b,c)}$ be an irreducible representation of $\mathfrak{spin}(7)$ with highest weight $(a,b,c)$ and $V_{(a,b)}$ be an irreducible representation of $\mathfrak{g}_2$ with highest weight $(a,b)$. The Casimir operators can be written as $\rho_{(a,b,c)}\left(\Cas_{\mathfrak{spin}(7)}\right)= c^{\mathfrak{spin}(7)}_{(a,b,c)}\Id, \rho_{(a,b)}(\Cas_{\mathfrak{g}_2})= c^{\mathfrak{g}_2}_{(a,b)}\Id$, where the Casimir eigenvalues are given by,
\begin{align}
	c^{\mathfrak{g}_2}_{(a,b)}&=-\frac{8}{9}(a^2+3b^2+3ab+5a+9b),\label{caseig}\\
	c^{\mathfrak{spin}(7)}_{(a,b,c)}&= -\frac{1}{3}(4a^2+8b^2+3c^2+8ab+8bc+4ca+20a+32b+18c).\label{caseig2}
\end{align}
These expressions differ from that of \cite{singhal2021paper}, because we use a different normalisation of the Casimir operator and an opposite convention for the order of $a,b,c$.
\subsection{Eigenvalues of the Twisted Dirac Operator}
For an FNFN $Spin(7)$-instanton, the fastest rate of convergence is $-2$. Hence we consider the family of moduli spaces $\mathcal{M}(A_\Sigma, \nu)$ for $\nu \in (-2,0)$. We want to find the critical weights in $(-2,0)$, i.e., $\nu \in (-2,0)$ such that $\nu + \frac{5}{2} \in \Spec\slashed{\mathfrak{D}}_{A_\Sigma}$. Hence, we are interested in finding all the eigenvalues of the twisted Dirac operator in the interval $\left[-\frac{5}{2}, \frac{5}{2}\right]$.\par 
Since $\mathfrak{spin}(7)_\mathbb{C} = V_{(0,1,0)} = V_{(1,0)}\oplus V_{(0,1)}$, we have
$$\Hom\left(V_\gamma, \Delta\otimes\mathfrak{spin}(7)_\mathbb{C}\right)^{G_2} = \Hom\left(V_\gamma, \Delta\otimes V_{(1,0)}\right)^{G_2} \oplus \Hom\left(V_\gamma, \Delta\otimes V_{(0,1)}\right)^{G_2}.$$
Then, from (\ref{diag1/3}) and using $c^{\mathfrak{g}_2}_{(1,0)} = -48/9$ and $c^{\mathfrak{g}_2}_{(0,1)} = -96/9$, we calculate the eigenvalues of $\left(\slashed{\mathfrak{D}}_{A_\Sigma}^{1/3}\right)_\gamma^2$ and their multiplicities. The eigenvalues are $-c^\mathfrak{spin(7)}_\gamma+\frac{1}{9}$ and $-c^\mathfrak{spin(7)}_\gamma-\frac{47}{9}$ with multiplicities $\dim\Hom\left(V_\gamma, \Delta\otimes V_{(1,0)}\right)^{G_2}$ and $\dim\Hom\left(V_\gamma, \Delta\otimes V_{(0,1)}\right)^{G_2}$ respectively.\par
We have the following important corollary, which follows from Theorem \ref{eigbnd} and Equation (\ref{caseig2}).
\begin{cor}\label{6eig}
	Consider the irreducible representations of $Spin(7)$ given by
		$$V_{(0,0,0)},\ V_{(1,0,0)},\ V_{(0,0,1)},\ V_{(0,1,0)},\ V_{(2,0,0)},\ V_{(1,0,1)}.$$
	If $V_\gamma$ is not one of these irreducible representations, then the operator
	$$\left(\slashed{\mathfrak{D}}_{A_\Sigma}^0\right)_\gamma : \Hom\left(V_\gamma, \Delta\otimes\mathfrak{spin}(7)_\mathbb{C}\right)^{G_2} \to \Hom\left(V_\gamma, \Delta\otimes\mathfrak{spin}(7)_\mathbb{C}\right)^{G_2}$$
	has no eigenvalues in the interval $\left(-\frac{5}{2}, \frac{5}{2}\right)$.
\end{cor}
Now, we describe the outline of the method to calculate the eigenvalues of the twisted Dirac operator corresponding to the representations mentioned in Corollary \ref{6eig}.\par 
Let $V_\gamma$ be an irreducible representation of $Spin(7)$. We want to find the matrix of the operator given in (\ref{dirac-t-gamma}). We note that, by $G_2$-equivariance, on $\Hom(V_{(0,0,1)} \otimes V_\gamma, \mathfrak{spin}(7)_\mathbb{C})^{G_2}$, we have
\begin{equation}\label{g2eqivrel}
    \rho_{\Delta\otimes V_\gamma}(\Cas_{\mathfrak{g}_2}) = \rho_{\mathfrak{spin}(7)_\mathbb{C}}(\Cas_{\mathfrak{g}_2}).
\end{equation}
Hence, from (\ref{diracrho}), (\ref{g2eqivrel}) and the isomorphism $V_\gamma^* \cong V_\gamma$, we can rewrite the operator $\left(\slashed{\mathfrak{D}}_{A_\Sigma}^\rho\right)_\gamma$  as
\begin{equation}\label{diracrhogamma}
    \left(\slashed{\mathfrak{D}}_{A_\Sigma}^\rho\right)_\gamma = \frac{1}{2}\left(\rho_{\Delta \otimes V_\gamma}(\Cas_{\mathfrak{spin}(7)})-\rho_{\mathfrak{spin}(7)_\mathbb{C}}(\Cas_{\mathfrak{g}_2})-\rho_\Delta(\Cas_{\mathfrak{m}})-\rho_{V_\gamma}(\Cas_{\mathfrak{m}})\right).
\end{equation}
First, we want to find a basis of $\Hom(V_{(0,0,1)} \otimes V_\gamma, \mathfrak{spin}(7)_\mathbb{C})^{G_2}$ that diagonalizes\\ $\rho_{\mathfrak{spin}(7)_\mathbb{C}}(\Cas_{\mathfrak{g}_2})+\rho_\Delta(\Cas_{\mathfrak{m}})+\rho_{V_\gamma}(\Cas_{\mathfrak{m}})$. We construct the basis by non-zero $G_2$-equivariant maps
\begin{equation}\label{q-basis}
    q^{(i,j)(k,l)}_{(m,n)} : V_{(0,0,1)} \otimes V_\gamma \to V_{(i,j)} \otimes V_{(k,l)} \to V_{(m,n)} \to \mathfrak{spin}(7)_\mathbb{C},
\end{equation}
where $V_{(m,n)}$ is either $V_{(1,0)}$ or $V_{(0,1)}$. We identify $\mathfrak{spin}(7)_\mathbb{C}$ with $\Lambda^2(\mathbb{C}^7)$. We use various identities and the projection formulae to write down explicit expressions of these maps. Then, we want to find a basis that diagonalizes $\rho_{\Delta \otimes V_\gamma}(\Cas_{\mathfrak{spin}(7)})$. We consider the maps
\begin{equation}\label{p-basis}
    p^{(i,j,k)}_{(m,n)} : V_{(0,0,1)} \otimes V_\gamma \to V_{(i,j,k)} \to V_{(m,n)} \to \mathfrak{spin}(7)_\mathbb{C}.
\end{equation}
Then, $p^{(i,j,k)}_{(m,n)}$ are eigenvectors of $\rho_{\Delta \otimes V_\gamma}(\Cas_{\mathfrak{spin}(7)})$ with eigenvalues $c^{\mathfrak{spin}(7)}_{(i,j,k)}$. From the explicit expressions of $q$-basis and $p$-basis elements, we write $p^{(i,j,k)}_{(m,n)}$ in terms of $q^{(i,j)(k,l)}_{(m,n)}$ and the change of basis matrix. Now, $q^{(i,j)(k,l)}_{(m,n)}$ are eigenvectors of $-\rho_{\mathfrak{spin}(7)_\mathbb{C}}(\Cas_{\mathfrak{g}_2})-\rho_\Delta(\Cas_{\mathfrak{m}})-\rho_{V_\gamma}(\Cas_{\mathfrak{m}})$ with eigenvalues $c^{\mathfrak{g}_2}_{(i,j)}+c^{\mathfrak{g}_2}_{(k,l)}-c^{\mathfrak{g}_2}_{(m,n)}-c^{\mathfrak{spin}(7)}_{(0,0,1)}-c^{\mathfrak{spin}(7)}_\gamma$ and $p^{(i,j,k)}_{(m,n)}$ are eigenvectors of $\rho_{\Delta \otimes V_\gamma}(\Cas_{\mathfrak{spin}(7)})$ with eigenvalues $c^{\mathfrak{spin}(7)}_{(i,j,k)}$. Then using the change of basis matrix, we write the matrix of $\left(\slashed{\mathfrak{D}}_{A_\Sigma}^\rho\right)_\gamma$ in the $q$-basis (\ref{diracrhogamma}).\par 
Next, we calculate the matrix of $M_\gamma$ in the $q$-basis. From (\ref{m-gamma}), we see that it is a diagonal matrix with entries either $1$ or $-1$, since $q_i$ factors through either $V_{(0,0)} \cong \Lambda^0 \subset \Delta$, or $V_{(1,0)} \cong \Lambda^1 \subset \Delta$. Then, we calculate the matrix of $\left(\slashed{\mathfrak{D}}_{A_\Sigma}^{\widetilde{\rho}}\right)_\gamma$ in the $q$-basis (\ref{diracrhotil}). In the $q$-basis, we have $\phi$ acting as a diagonal matrix with entries either $7$ or $-1$, by Lemma \ref{eigenphi}, since $q_i$ factors through either $V_{(0,0)} \cong \Lambda^0 \subset \Delta$, or $V_{(1,0)} \cong \Lambda^1 \subset \Delta$. Consequently, using (\ref{dirac-t-gamma}), we calculate the matrix of $\left(\slashed{\mathfrak{D}}_{A_\Sigma}^t\right)_\gamma$ in the $q$-basis. We note that for $t = 1/3$, by (\ref{diag1/3}), $\left(\slashed{\mathfrak{D}}_{A_\Sigma}^{1/3}\right)_\gamma^2$ should be a diagonal matrix in the $q$-basis, where the entries are either $-c^\mathfrak{spin(7)}_\gamma+\frac{1}{9}$ or $-c^\mathfrak{spin(7)}_\gamma-\frac{47}{9}$, which acts as a consistency check for our calculations. Throughout the calculations, we use (\ref{caseig}) and (\ref{caseig2}) to calculate the Casimir eigenvalues. Finally, for $t = 0$ in the matrix of $\left(\slashed{\mathfrak{D}}_{A_\Sigma}^t\right)_\gamma$ in the $q$-basis, we calculate the desired eigenvalues of $\left(\slashed{\mathfrak{D}}_{A_\Sigma}^0\right)_\gamma$. 
\begin{thm}\label{eigdir}
The eigenvalues of the twisted Dirac operator $\left(\slashed{\mathfrak{D}}_{A_\Sigma}^0\right)_\gamma$ are 
    \begin{enumerate}
        \item $\frac{1}{2}$ corresponding to $V_\gamma = V_{(0,0,0)}$,
        \item $\frac{1}{6}(-3+2\sqrt{105}),\ \frac{1}{6}(-3-2\sqrt{105}),\ -\frac{3}{2}$ corresponding to $V_\gamma = V_{(1,0,0)}$,
        \item $\frac{1}{6}(-3-8\sqrt{6}),\ \frac{1}{6}(-3+8\sqrt{6}),\ -\frac{5}{2},\ \frac{3}{2}$ corresponding to $V_\gamma = V_{(0,0,1)}$,
        \item $\frac{1}{2}(-1-2\sqrt{17}),\ \frac{1}{6}(-3-2\sqrt{105}),\frac{1}{2}(-1+2\sqrt{17}),\ -\frac{7}{2}, \ \frac{1}{6}(-3+2\sqrt{105})$ corresponding $V_\gamma = V_{(0,1,0)}$,
        \item $-\frac{25}{6},\ \frac{23}{6}$ corresponding to $V_\gamma = V_{(2,0,0)}$,
        \item $\frac{1}{2}(-1-4\sqrt{5}),\ \frac{1}{6}(-3-4\sqrt{33}),\ \frac{1}{6}(1+4\sqrt{37}),\ \frac{1}{2}(-1+4\sqrt{5}), \ \frac{1}{6}(1-4\sqrt{37}), \ \frac{1}{6}(-3+4\sqrt{33}),\ -\frac{19}{6}$ corresponding to $V_\gamma = V_{(1,0,1)}$.

    \end{enumerate}
\end{thm}
\begin{cor}\label{eigdircor}
    The eigenvalues of the twisted Dirac operator $\left(\slashed{\mathfrak{D}}_{A_\Sigma}^0\right)_\gamma$ in the interval $\left[-\frac{5}{2}, \frac{5}{2}\right]$ are $-\frac{5}{2},-\frac{3}{2}, \frac{1}{2}, \frac{3}{2}$, and that of in the interval $\left(\frac{1}{2}, \frac{5}{2}\right)$ is $\frac{3}{2}$ corresponding to the spin representation $V_{(0,0,1)}$.
\end{cor}

\section{The Space of Deformations of FNFN \texorpdfstring{$Spin(7)$}{Spin(7)}-Instanton}\label{section5}
In this section, we shall study the deformations of AC $Spin(7)$-instantons on the AC $Spin(7)$-manifold $\mathbb{R}^8$, where the $Spin(7)$-instantons on $\mathbb{R}^8$ will converge to the canonical connection on $S^7$ at infinity. Fairlie--Nuyts \cite{fairlie1984paper} and Fubini--Nicolai \cite{fubini1985paper} independently constructed these instantons on $\mathbb{R}^8$, and hence will be referred to as \define{FNFN $Spin(7)$-instanton} on $\mathbb{R}^8$.
\subsection{FNFN \texorpdfstring{$Spin(7)$}{Spin(7)}-Instanton}
In this subsection, we introduce FNFN-instanton on $\mathbb{R}^8$. For a derivation of the instanton using homogeneous space techniques, see \cite{Papo2022instantons}.\par
Let us consider the asymptotically conical $Spin(7)$-manifold $\mathbb{R}^8$ asymptotic to the nearly $G_2$ manifold $\Sigma = S^7$. We consider $S^7$ as a homogeneous nearly $G_2$ manifold  $Spin(7)/G_2$. Then we have the canonical bundle $G_2 \to Spin(7) \to S^7$ (call this bundle $P$). Also consider the bundle $Spin(7) \to Spin(7) \times_{(G_2, \iota)}Spin(7) \to S^7$ (call this bundle $Q$) where $\iota : G_2 \hookrightarrow Spin(7)$ is the inclusion. This bundle is (bundle) isomorphic to the trivial bundle $Spin(7) \to Spin(7) \times S^7 \to S^7$.\par 
Now the basis $I_A$  for $\mathfrak{spin}(7)$ can be represented by left invariant vector fields $\widehat{E}_A$ on $Spin(7)$ and also by the dual basis $\hat{e}^A$ of left invariant $1$-forms. Denote the natural projection map
\begin{align*}
	\pi : Spin(7) &\to Spin(7)/G_2\\
	g &\mapsto gG_2
\end{align*}
of the principal bundle. Let $U$ be a contractible open subset of $Spin(7)/G_2$. Then we define the map $L : U \to Spin(7)$ such that $\pi \circ L = \Id_U$, i.e., $L$ is a local section of the bundle $Spin(7) \to Spin(7)/G_2$. We put $e^A := L^*\hat{e}^A$. Then $\{e^a : a = 1, \dots, 7\}$ form an orthonormal frame for $T^*(Spin(7)/G_2)$ over $U$.\par
Consider the $8$-dimensional manifold $\mathbb{R}\times Spin(7)/G_2$. We choose the metric $g_8 = (e^0)^2+g_7$ where $e^0 = dt$ for $t = \ln r$ the coordinate of $\mathbb{R}$, and $g_7$ is the metric on $Spin(7)/G_2$. This metric is conformal to the flat metric on $\mathbb{R}^8$. The connection $1$-form is given by $A = A_0e^0+A_ae^a$ which gives the $Spin(7)$-invariant connection
\begin{equation}\label{fnfnins}
    A = e^iI_i + \varphi(t)e^aI_a
\end{equation}
where $e^iI_i$ is the canonical connection $A_\Sigma$, we have taken $A_0 = 0$ (\define{temporal gauge}), and $\varphi$ is given by 
\begin{equation}\label{varphiins}
    \varphi = \frac{1}{1+e^{2t+2C_1}}.
\end{equation}
The curvature is given by
\begin{equation}\label{fnfncurv}
    F_{bc} = (\varphi^2-1)f^i_{bc}I_i+(\varphi^2-\varphi)f_{abc}I_a,\ \ F_{0a} = \dot{\varphi}I_a.
\end{equation}
The connection $A$ defined in (\ref{fnfnins}), where $\varphi$ is given in (\ref{varphiins}), is an instanton on $\mathbb{R}^8$. We call this the \define{FNFN $Spin(7)$-instanton}. Clearly FNFN $Spin(7)$-instanton $A$ is asymptotic to the canonical connection $A_\Sigma$ with fastest rate of convergence $-2$, since $\varphi = O(r^{-2})$ as $r\to\infty$.
\subsection{Index of the Twisted Dirac Operator}
We want to calculate the index of the Dirac operator $\slashed{\mathfrak{D}}_A^-$ on $\slashed{S}(\mathbb{R}^8)$ twisted by the trivial bundle $\mathfrak{g}_P := \mathfrak{spin}(7) \times \mathbb{R}^8$ over $\mathbb{R}^8$. We use the \textit{Atiyah--Patodi--Singer Index Theorem} for manifolds with boundaries, by relating the index of the Dirac operator $\slashed{\mathfrak{D}}_A^-$ on $\mathbb{R}^8$ with the index of the Dirac operator on a closed ball $B^8_R$ of large enough radius $R$. Moreover, we consider the FNFN instanton to be an instanton on $\mathbb{R}^8$ and, for the purposes of calculating the eta-invariant appearing in the index theorem, on $\mathbb{R} \times S^7$.\par
The Atiyah--Patodi--Singer index theorem is applicable when the manifold has non-empty boundary (for more details see \cite{Atiyah-patodi1975:1}, \cite{Atiyah-patodi1976:2}, \cite{Atiyah-patodi1976:3}, \cite{eguchi1980}, \cite{galkey1981}).
Let $M$ be a $8$-manifold with non-empty boundary $\partial M$. Let $P \to M$ be a principal $G$-bundle and consider the negative Dirac operator $\slashed{\mathfrak{D}}_A^-$ acting on the bundle $E := \slashed{S}^-(X) \otimes \mathfrak{g}_P$ over $M$. Then the index of the operator $\slashed{\mathfrak{D}}_A^-$ requires topological information on the manifold $M$ as well as analytic information on the boundary $\partial M$.\par 
The Atiyah--Patodi--Singer index theorem for a manifold $M$ with non-empty boundary $\partial M$ is given by
\begin{equation}\label{APS}
    \Ind (\slashed{\mathfrak{D}}_A^-, M, \partial M) = I(M) + CS(\partial M) + \frac{1}{2}\eta (\partial M)
\end{equation}
where $I(M)$ is an integral of characteristic classes over $M$ and $\eta (\partial M)$ is the eta-invariant of the boundary. The Chern--Simons term $CS(\partial M)$ of the boundary arises when the manifold does not admit a product metric on the boundary. Moreover, the Dirac operator $\slashed{\mathfrak{D}}_A^-$ is subject to non-local boundary condition which will be explained later.

\subsubsection{\texorpdfstring{$\mathbb{R}^8$}{R8} with Cigar Metric and Index}
Let $g_C$ be the asymptotically conical metric, i.e., the flat metric on $\mathbb{R}^8$. We define the metric $g_{CI} := \frac{1}{\varrho^2}g_C$, where $\varrho$ is the radius function (\ref{radiusfn}). Then $(\mathbb{R}^8, g_{CI})$ resembles a cigar (the reason $g_{CI}$ is usually called a \textit{cigar metric}). In particular, for 
$$\varrho(r) = \begin{cases}
    r &r>1\\
    \frac{1}{2}(1+r^2) &r<1,
\end{cases}$$
$(\mathbb{R}^8, g_{CI})$ is a hemisphere $M_1$ glued to a cylinder $M_2 = (1, \infty) \times S^7$.\par
The weighted Sobolev space $W^{0,2}_\nu$ on $\mathbb{R}^8$ is defined by the norm
$$\|\eta\|_{W^{0,2}_\nu} := \left(\int|\varrho^\nu\eta|^2\varrho^{-8}\dvol_C\right)^{1/2}.$$
The space $W^{0,2}_{CI}$ of $L^2$-functions on the cigar $M$ is defined by
$$\|\eta\|_{W^{0,2}_{CI}} := \left(\int|\eta|^2\dvol_{CI}\right)^{1/2}.$$
Now, $\dvol_{CI} = \varrho^{-8}\dvol_C$. Hence,
\begin{align*}
	W^{0,2}_\nu &\to W^{0,2}_{CI}\\
	\eta &\mapsto \varrho^\nu\eta
\end{align*}
is an isomorphism. Similarly, we can extend this to an isomorphism $W^{k,2}_\nu \to W^{k,2}_{CI}$.\par 
By conformal properties of Dirac operators, we have that the Dirac operator of $g_{CI}$ is $\slashed{\mathfrak{D}}_{A,CI}^- = \varrho^{\frac{9}{2}}\slashed{\mathfrak{D}}_{A,C}^-\varrho^{-\frac{7}{2}}$ (where $\slashed{\mathfrak{D}}_{A,{CI}}^-$ and $\slashed{\mathfrak{D}}_{A,C}^-$ are the Dirac operators corresponding to cigar and conical metrics respectively.) Then we have the commutative diagram
\[ \begin{tikzcd}
	W^{k,2}_{CI} \arrow{rr}{\slashed{\mathfrak{D}}_{A,{CI}}^-} \arrow[swap]{dd}{\varrho^{-\frac{7}{2}}}& & W^{k-1,2}_{CI}  \\%
	\\
	W^{k,2}_{-\frac{7}{2}} \arrow{rr}{\slashed{\mathfrak{D}}_{A,C}^-} & & W^{k-1,2}_{-\frac{9}{2}}.\arrow[swap]{uu}{\varrho^{\frac{9}{2}}}
\end{tikzcd}
\]
Since the vertical arrows are isomorphism, we have
\begin{equation}\label{indrate}
    \ind\left(\slashed{\mathfrak{D}}_{A,C}^- : W^{k,2}_{-\frac{7}{2}} \to W^{k-1,2}_{-\frac{9}{2}}\right) = \ind\left(\slashed{\mathfrak{D}}_{A,{CI}}^- : W^{k,2}_{CI} \to W^{k-1,2}_{CI}\right).
\end{equation}
Let us define a function $\widetilde{\varphi} : \mathbb{R} \to \mathbb{R}$ by
\begin{equation}\label{phitilde'}
    \widetilde{\varphi}(t) = \begin{cases}
    1 &t < -T\\
    \alpha', &-T < t < -\frac{T}{2}\\
    \varphi(t), &-\frac{T}{2} < t < \frac{T}{2}\\
    \alpha, &\frac{T}{2} < t < T\\
    0 &t > T
    \end{cases}
\end{equation}
where $\alpha$ is a smooth interpolation between its definition at $\frac{T}{2}$ and $T$ and $\alpha'$ is that of between its definition at $-T$ and $-\frac{T}{2}$.\par 
Then we have a connection
\begin{equation}\label{atildeclass}
    \widetilde{A} = A_\Sigma + \widetilde{\varphi}(t)e^aI_a.
\end{equation}
The following proposition enables us to reduce the calculation of the index to a case of compact manifold with APS boundary conditions.
\begin{prop}\label{indexequal}
	Let $B^8_R := \{x \in \mathbb{R}^8 : |x| \leq R\}$ be $8$-dimensional ball of radius $R$. Then for sufficiently large $R$, we have
	$$\ind\left(\slashed{\mathfrak{D}}_{A,{CI}}^-, \mathbb{R}^8, g_{CI}\right) = \ind\left(\slashed{\mathfrak{D}}_{\widetilde{A},{CI}}^-, B^8_R, g_{CI}\right).$$
 Moreover, for sufficiently large $T$, we have
 $$\ind\left(\slashed{\mathfrak{D}}_A^-, \mathbb{R} \times S^7, g\right) = \ind\left(\slashed{\mathfrak{D}}_{\widetilde{A}'}^-, [-T,T] \times S^7, g\right)$$
 where $g$ is the cylindrical metric $g = dt^2 +g_{S^7}$.
\end{prop}
\begin{proof}
	Let $\eta : \mathbb{R}^8 \to \slashed{S}(\mathfrak{g}_P)$ be a spinor such that $\slashed{\mathfrak{D}}_{\widetilde{A},{CI}}^-\eta = 0$ and $\eta \in L^2(\mathbb{R}^8, g_{CI})$. Now,\\
    $\slashed{\mathfrak{D}}_{\widetilde{A},{CI}}^- = E^0 \left(\frac{\partial}{\partial t} - \slashed{\mathfrak{D}}_{\widetilde{A}_{t, \Sigma}}\right)$. For $t > \ln R$, since $\widetilde{\varphi}(t) = 0$, we have $\slashed{\mathfrak{D}}_{\widetilde{A}_{t, \Sigma}} = \slashed{\mathfrak{D}}_{A_{\Sigma}}$. Let $\lambda_n \in \Spec\slashed{\mathfrak{D}}_{A_{\Sigma}}$. Then, we have the Fourier expansion of $\eta$ given by $\eta = \sum\limits_{n \in \mathbb{Z}}e^{\lambda_n(t - \ln R)}\eta_n$ where $\eta_n \in \ker\left(\slashed{\mathfrak{D}}_{A_{\Sigma}} - \lambda_n\right)$. Hence, $\eta \in L^2$ implies $\eta_n = 0$ when $\lambda_n > 0$. So $\eta$ can be written as a sum of eigenvectors $\eta_n$ of Dirac operator on the boundary with negative eigenvalues. Hence $\eta$ solves \textit{Atiyah--Patodi--Singer boundary condition}.\par 
	Conversely,
	let $\eta : B^8_R \to \slashed{S}(\mathfrak{g}_P)$ such that $\slashed{\mathfrak{D}}_{\widetilde{A},{CI}}^-\eta = 0$ and $\eta$ solves Atiyah--Patodi--Singer boundary condition. We extend $\eta$ to $\mathbb{R}^8$. On $\partial B^8_R$, $\eta = \sum\limits_{n < 0}\eta_n$, where $\lambda_n < 0$ if and only if $n < 0$. So, for $r > R$ (i.e., $t > \ln R$) we set $\eta = \sum\limits_{n < 0}e^{\lambda_n(t - \ln R)}\eta_n$. Then $\eta \in L^2(\mathbb{R}^8, g_{CI})$ and solves $\slashed{\mathfrak{D}}_{\widetilde{A},{CI}}^-\eta = 0$.\par 
	Hence, we have just proved that $\ker\left(\slashed{\mathfrak{D}}_{\widetilde{A},{CI}}^-, \mathbb{R}^8, g_{CI}\right) \cong \ker\left(\slashed{\mathfrak{D}}_{\widetilde{A},{CI}}^-, B^8_R, g_{CI}\right)$. Similarly, we can show that $\coker\left(\slashed{\mathfrak{D}}_{\widetilde{A},{CI}}^-, \mathbb{R}^8, g_{CI}\right) \cong \coker\left(\slashed{\mathfrak{D}}_{\widetilde{A},{CI}}^-, B^8_R, g_{CI}\right)$. Hence, we proved 
    $$\Ind\left(\slashed{\mathfrak{D}}_{\widetilde{A},{CI}}^-, \mathbb{R}^8, g_{CI}\right) = \Ind\left(\slashed{\mathfrak{D}}_{\widetilde{A},{CI}}^-, B^8_R, g_{CI}\right).$$
    Finally, we prove that $\Ind\left(\slashed{\mathfrak{D}}_{A,{CI}}^-, \mathbb{R}^8, g_{CI}\right) = \Ind\left(\slashed{\mathfrak{D}}_{\widetilde{A},{CI}}^-, \mathbb{R}^8, g_{CI}\right)$. Consider 
    \begin{align}\label{diracpert}
        \|\slashed{\mathfrak{D}}_{A,{CI}}^- - \slashed{\mathfrak{D}}_{\widetilde{A},{CI}}^-\| &= \|(\varphi(t) - \widetilde{\varphi}(t))e^aI_a\| = \sup_{\eta \in L^2(\mathbb{R}^8, g_{CI})}\frac{\|(\varphi(t) - \widetilde{\varphi}(t))e^aI_a\eta\|_{L^2}}{\|\eta\|_{L^2}}.
    \end{align}
    Now,
    $$\|(\varphi(t) - \widetilde{\varphi}(t))e^aI_a\eta\|_{L^2}^2 \leq \sup (\varphi(t) - \widetilde{\varphi}(t))^2\|e^aI_a\|^2\|\eta\|_{L^2}^2.$$
    Thus, from (\ref{diracpert}), we have $\|\slashed{\mathfrak{D}}_{A,{CI}}^- - \slashed{\mathfrak{D}}_{\widetilde{A},{CI}}^-\| \leq  \sup |\varphi(t) - \widetilde{\varphi}(t)|\|e^aI_a\|$. Hence, for all $\epsilon > 0$, there exists $R >0$ such that $\|\slashed{\mathfrak{D}}_{A,{CI}}^- - \slashed{\mathfrak{D}}_{\widetilde{A},{CI}}^-\| < \epsilon$. Then the result follows from the fact that two Fredholm operators belonging to the same connected component of the space of all Fredholm operators have the same index, since the Fredholm index is continuous and integer-valued.\par
    The second part of the theorem can be proved similarly.
\end{proof}

\subsubsection{The Term \texorpdfstring{$I(\mathbb{R}^8)$}{I(R8)}}
The term $I(\slashed{\mathfrak{D}}_{A,{CI}}^-, \mathbb{R}^8, g_{CI})$ in (\ref{APS}) is given by
\begin{align*}
	I(\slashed{\mathfrak{D}}_{A,{CI}}^-, \mathbb{R}^8, g_{CI}) &= - \int_{\mathbb{R}^8}\widehat{A}(M)\ch(\mathfrak{g}_P \otimes \mathbb{C})\\
	&= -\frac{1}{12}\int_{\mathbb{R}^8}\left(p_1(\mathfrak{g}_P)^2 - 2p_2(\mathfrak{g}_P)\right) + \frac{1}{24}\int_{M_1}p_1(M_1)p_1(\mathfrak{g}_P) + \frac{1}{24}\int_{M_2}p_1(M_2)p_1(\mathfrak{g}_P) \\
	&\hspace{0.5cm}- \frac{1}{5760}\dim\mathfrak{g}\int_{M_1}(7p_1(M_1)^2 - 4p_2(M_1)) - \frac{1}{5760}\dim\mathfrak{g}\int_{M_2}(7p_1(M_2)^2 - 4p_2(M_2)),
\end{align*}
where the Pontryagin classes $p_i$ are given in terms of the curvature as
$$p_1(\mathfrak{g}_P) = - \frac{1}{8\pi^2}\tr(F_A^2),\ \ \ p_2(\mathfrak{g}_P) = \frac{1}{128\pi^4}\left[\tr(F_A^2)^2 - 2 \tr F_A^4\right],$$
where the trace is taken over $\mathfrak{g}$.\par
Simple calculation shows that $p_1(M_1) = 0,\ \ p_2(M_1) = 0$, and
$p_1(M_2) = 0,\ \ p_2(M_2) = 0$. Hence, 
\begin{equation}\label{Iterm}
    I(\slashed{\mathfrak{D}}_{A,{CI}}^-, \mathbb{R}^8, g_{CI}) = -\frac{1}{12}\int_{\mathbb{R}^8}\left(p_1(\mathfrak{g}_P)^2 - 2p_2(\mathfrak{g}_P)\right) = -\frac{1}{384\pi^4}\int_{\mathbb{R}^8}\tr F_A^4.
\end{equation}
\subsubsection{Eta Invariant of the Boundary}
We calculate the eta-invariant of the twisted Dirac operator by relating it to the untwisted Dirac operator, whose eta-invariant is zero, using a spectral flow.\par
Recall the FNFN $Spin(7)$-instanton $A$ given by (\ref{fnfnins}),
where $\varphi(t)$ is given by (\ref{varphiins}), can be identified with a family of connections $\{A_t : t \in \mathbb{R}\}$ on $S^7$. Then, we have a family of Dirac operators on $S^7$ twisted by the connections $A_t$ given by
$$\slashed{\mathfrak{D}}_{A_{t, \Sigma}} = \slashed{\mathfrak{D}}_{A_\Sigma} + \varphi(t)e^aI_a.$$
Now, the curvature of the connection is given by (\ref{fnfncurv}) for which we note that $F_{bc} = 0$ for $\varphi(t) = 1$. Hence, $A_t$ is a flat connection for $t = -\infty$. Since the underlying manifold is simply connected, this flat connection is the trivial connection. Hence corresponding to this connection, or equivalently, for $\varphi(t) = 1$, we have the untwisted Dirac operator $\mathcal{D}_{\Sigma}$, i.e.,
\begin{equation}
    \mathcal{D}_{\Sigma} = \slashed{\mathfrak{D}}_{A_\Sigma} + e^aI_a.
\end{equation}
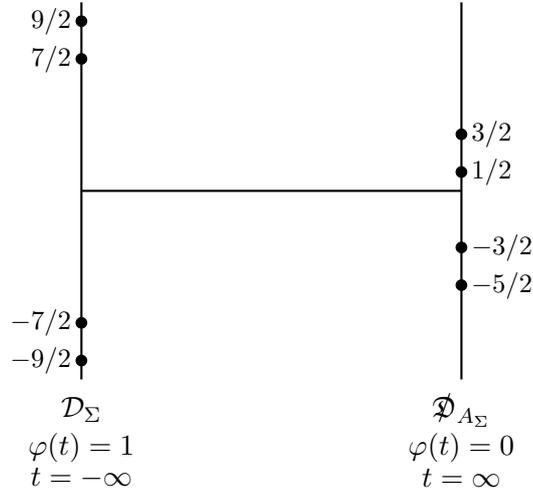
We want to calculate the \define{spectral flow} of the family $\left\{\slashed{\mathfrak{D}}_{A_{t, \Sigma}}\right\}_{t \in \mathbb{R}}$, where spectral flow is the net number of eigenvalues flowing from negative to positive. First let us calculate the eigenvalues of the operator $e^aI_a$. We note that the operator $e^aI_a$ acts fibre-wise: on $\Delta \otimes \mathfrak{spin}(7)$. Let $e^\mu,\ \mu = 0, 1, \dots, 7$ be a basis of $\Delta$ and $I_A$ be a basis of $\mathfrak{spin}(7)$. Then,
$$\left(e^aI_a\right)\left(e^\mu \otimes I_A\right) = (e^a \cdot e^\mu) \otimes [I_a, I_A] = (E^a \otimes \ad I_a)(e^\mu, I_A),$$
where $E^a$ is the matrix of Clifford multiplication with $e^a$, calculated using (\ref{cliffmult1}). Taking the Kronecker product of $E^a$ and $\ad I_a$, we get the matrix of $e^aI_a$ whose eigenvalues are given by $4, -4, \pm\frac{1}{3}(-3 + \sqrt{33}), \pm \frac{1}{3}(-1 + \sqrt{57}), 2, 0$ with multiplicities $8,7,14,27,7$ and $64$ respectively.\par 
Now, near zero, the eigenvalues of $\mathcal{D}_{\Sigma}$ are given by $\pm 7/2, \pm 9/2$ (see \cite{bourguignon2015}) and that of $\slashed{\mathfrak{D}}_{A_\Sigma}$ are $-5/2, -3/2, 1/2, 3/2$ (from Corollary \ref{eigdircor}). We note that the eigenvalue of $e^aI_a$ with the highest magnitude is $4$. Hence the only possibility of having a non-zero spectral flow is the eigenvalue $1/2$ of $\slashed{\mathfrak{D}}_{A_\Sigma}$ flowing down to the eigenvalue $-7/2$ of $\mathcal{D}_{\Sigma}$. Since $1/2$ corresponds to eigenvalue of $\slashed{\mathfrak{D}}_{A_\Sigma}$ obtained from the trivial representation $V_{(0,0,0)}$ of $Spin(7)$, the eigenspinor $\eta$ corresponding to eigenvalue $1/2$ belongs to the space $\Hom(V_{(0,0,0)}, \Delta \otimes \mathfrak{spin}(7)_\mathbb{C})^{G_2} \otimes V_{(0,0,0)} \subset L^2(Spin(7), \Delta \otimes \mathfrak{spin}(7)_\mathbb{C})^{G_2}$ in the decomposition (\ref{frob2}). Now, we have the decomposition
$$\Delta \otimes \mathfrak{spin}(7)_\mathbb{C} \cong V_{(0,0,1)} \otimes V_{(0,1,0)} \cong V_{(0,0)} \oplus 3V_{(1,0)} \oplus 2V_{(0,1)} \oplus 2V_{(2,0)} \oplus V_{(1,1)}.$$
Hence by Schur's lemma, we have that $\eta \in \Hom(V_{(0,0,0)}, V_{(0,0)})^{G_2} \otimes V_{(0,0,0)}$ which is a subspace of $L^2(Spin(7), V_{(0,0)})^{G_2}$. Hence, in order to check whether a flow from the eigenvalue $1/2$ of $\slashed{\mathfrak{D}}_{A_\Sigma}$ flowing down to the eigenvalue $-7/2$ of $\mathcal{D}_{\Sigma}$ exists, we need to calculate the eigenvalue of $e^aI_a$ corresponding to the trivial subrepresentation $V_{(0,0)}$ of $\Delta \otimes \mathfrak{spin}(7)_\mathbb{C}$. Then, since $\dim V_{(1,1)} = 64$, $V_{(1,1)}$ is the eigenspace of the eigenvalue $0$. Similarly, the two copies of $V_{(2,0)}$ are the eigenspaces of the eigenvalues $\frac{1}{3}(\mp1\pm\sqrt{57})$, the two copies of $V_{(0,1)}$ are the eigenspaces of the eigenvalues $\frac{1}{3}(\mp 3\pm\sqrt{33})$, the three copies of $V_{(1,0)}$ are the eigenspaces of the eigenvalues $2, 4$ and $-4$ respectively, $V_{(0,0)}$ is the eigenspace of one eigenvalue $4$. Thus we have a flow of the eigenvalue moving up to $9/2$ and not down to $-7/2$. Hence, there is no flow from the eigenvalue $1/2$ of $\slashed{\mathfrak{D}}_{A_\Sigma}$ to the eigenvalue $-7/2$ of $\mathcal{D}_{\Sigma}$, and hence, we have no flow of eigenvalues of $\mathcal{D}_{\Sigma}$ flowing up or down to the eigenvalues of $\slashed{\mathfrak{D}}_{A_\Sigma}$. Consequently, the spectral flow of the family $\left\{\slashed{\mathfrak{D}}_{A_{t, \Sigma}}\right\}_{t \in \mathbb{R}}$ is given by
\begin{equation}\label{spflow}
    \Sf\left(\left\{\slashed{\mathfrak{D}}_{A_{t, \Sigma}}\right\}_{t \in \mathbb{R}}\right) = 0.
\end{equation}
We recall that we can identify the family of Dirac operators $\left\{\slashed{\mathfrak{D}}_{A_{t, \Sigma}}\right\}_{t \in \mathbb{R}}$ on $S^7$ with a Dirac operator $\slashed{\mathfrak{D}}_A^-$ on the cylinder $\mathbb{R} \times S^7$, where the identification is given by
$$\slashed{\mathfrak{D}}_A^- = E^0 \cdot \left(\frac{d}{dt} - \slashed{\mathfrak{D}}_{A_{t, \Sigma}}\right).$$
Then, the index of the Dirac operator $\slashed{\mathfrak{D}}_A^-$ on the cylinder $\mathbb{R} \times S^7$ is precisely the negative of the spectral flow of the operator $\Sf\left(\left\{\slashed{\mathfrak{D}}_{A_{t, \Sigma}}\right\}_{t \in \mathbb{R}}\right)$ (see \cite{kronheimer-mrowka} proposition 14.2.1). This follows from the fact that $\frac{d}{dt}$ and $\slashed{\mathfrak{D}}_{A_{t, \Sigma}}$ have opposite signs, and Clifford multiplication by $E^0$ is an isomorphism that does not affect index. Hence, from (\ref{spflow}), we have
$$\Ind (\slashed{\mathfrak{D}}_A^-, \mathbb{R} \times S^7) = -\Sf\left(\left\{\slashed{\mathfrak{D}}_{A_{t, \Sigma}}\right\}_{t \in \mathbb{R}}\right) = 0.$$
Now, from Proposition \ref{indexequal} applying Atiyah--Patodi--Singer index formula on the compact manifold with boundary $[-T, T] \times S^7$, we have
$$\Ind (\slashed{\mathfrak{D}}_A^-, \mathbb{R} \times S^7) = \Ind (\slashed{\mathfrak{D}}_{\widetilde{A}}^-, [-T, T] \times S^7) = I\left(\slashed{\mathfrak{D}}_{\widetilde{A}}^-, [-T, T] \times S^7\right) + \frac{1}{2}\eta(\partial([-T, T] \times S^7)).$$
However, we note that $\Ind (\slashed{\mathfrak{D}}_A^-, \mathbb{R} \times S^7)$ is independent of $T$, and hence taking $T \to \infty$, we have
$$\Ind (\slashed{\mathfrak{D}}_A^-, \mathbb{R} \times S^7) = I\left(\slashed{\mathfrak{D}}_{\widetilde{A}}^-, \mathbb{R} \times S^7\right) + \frac{1}{2}\eta(\partial(\mathbb{R} \times S^7)).$$
Now, from (\ref{atildeclass}) and (\ref{Iterm}), we have
$$I\left(\slashed{\mathfrak{D}}_{\widetilde{A}}^-, \mathbb{R} \times S^7\right) = -\frac{1}{384\pi^4}\int_{\mathbb{R}^8}\tr F_{\widetilde{A}}^4 = -\frac{1}{384\pi^4}\int_{\mathbb{R}^8}\tr F_A^4 = I\left(\slashed{\mathfrak{D}}_A^-, \mathbb{R} \times S^7\right).$$
Moreover, since $\partial(\mathbb{R} \times S^7) = S^7 \amalg \overline{S^7}$, where $\overline{S^7}$ is $S^7$ with opposite orientation, we have
$$\eta(\partial(\mathbb{R} \times S^7)) = \eta(\mathcal{D}_{\Sigma}, \overline{S^7}) + \eta(\slashed{\mathfrak{D}}_{A_\Sigma}, S^7) = \eta(\slashed{\mathfrak{D}}_{A_\Sigma}, S^7) - \eta(\mathcal{D}_{\Sigma}, S^7) = \eta(\slashed{\mathfrak{D}}_{A_\Sigma}, S^7),$$
since, eta-invariant of $\mathcal{D}_{\Sigma}$ is zero, which follows from the fact that the metric and Levi-Civita connection of $S^7$ are invariant under an orientation-reversing isometry. We note that the orientation of $S^7$ corresponding to the operator $\slashed{\mathfrak{D}}_{A_\Sigma}$ is the same as the boundary $S^7$ of $\mathbb{R}^8$.\par
So, finally, we have
\begin{align}\label{etainv}
	\frac{1}{2}\eta(\slashed{\mathfrak{D}}_{A_\Sigma}, S^7) = \frac{1}{2}\eta(\partial(\mathbb{R} \times S^7)) = \Ind (\slashed{\mathfrak{D}}_A^-, \mathbb{R} \times S^7) - I(\slashed{\mathfrak{D}}_A^-, \mathbb{R} \times S^7) = \frac{1}{384\pi^4}\int_{\mathbb{R} \times S^7}\tr F_A^4.
\end{align}
\subsubsection{Index of the Twisted Dirac Operator}
From (\ref{indrate}) and Proposition \ref{indexequal}, we have
$$\Ind_{-\frac{5}{2}} (\slashed{\mathfrak{D}}_A^-, \mathbb{R}^8, g) = \Ind (\slashed{\mathfrak{D}}_{A,{CI}}^-, \mathbb{R}^8, g_{CI}) = \Ind (\slashed{\mathfrak{D}}_{\widetilde{A},{CI}}^-, B^8_R, g_{CI}).$$
Since, $B^8_R$ is a compact manifold with boundary, applying Atiyah--Patodi--Singer index formula, 
$$\Ind_{-\frac{5}{2}} (\slashed{\mathfrak{D}}_A^-, \mathbb{R}^8, g) = I\left(\slashed{\mathfrak{D}}_{\widetilde{A},{CI}}^-, B^8_R, g_{CI}\right) + \frac{1}{2}\eta(\slashed{\mathfrak{D}}_{A_\Sigma}, \partial B^8_R).$$
Since $\Ind_{-\frac{5}{2}} (\slashed{\mathfrak{D}}_A^-, \mathbb{R}^8, g)$ is independent of $R$, taking $R \to \infty$, and from (\ref{Iterm}) and (\ref{etainv}) we have
\begin{align}\label{ind-5/2}
	\Ind_{-\frac{5}{2}} (\slashed{\mathfrak{D}}_A^-, \mathbb{R}^8, g) &= I(\slashed{\mathfrak{D}}_{\widetilde{A},{CI}}^-, \mathbb{R}^8, g_{CI}) + \frac{1}{2}\eta(\slashed{\mathfrak{D}}_{A_\Sigma}, S^7)\nonumber\\
    &= -\frac{1}{384\pi^4}\int_{\mathbb{R}^8}\tr F_{\widetilde{A}}^4 + \frac{1}{384\pi^4}\int_{\mathbb{R} \times S^7}\tr F_A^4\nonumber\\
	&= -\frac{1}{384\pi^4}\int_{\mathbb{R}^8}\tr F_A^4 + \frac{1}{384\pi^4}\int_{\mathbb{R} \times S^7}\tr F_A^4\nonumber\\
    &= 0.
\end{align}
\subsection{The Main Result}
Finally, we have the main result on the deformations of FNFN $Spin(7)$-instanton.
\begin{thm}
	The virtual dimension of the moduli space $\mathcal{M}(A_\Sigma, \nu)$ of FNFN $Spin(7)$-instanton with decay rate $\nu \in (-2, 0) \setminus \{-1\}$ is given by
	\begin{align}
		\virtualdim \mathcal{M}(A_\Sigma, \nu) = \begin{cases}
			1 &\text{if }\ \nu \in (-2, -1)\\
			9 &\text{if }\ \nu \in (-1, 0).
		\end{cases}
	\end{align}
\end{thm}
\begin{proof}
    From (\ref{ind-5/2}) we have that the index of the Dirac operator $\slashed{\mathfrak{D}}_A^-$ corresponding to the rate $-5/2$ is zero. Moreover, from Corollary \ref{eigdircor}, we see that the only critical rates greater that $-5/2$ are $-2$ and $-1$, corresponding to the eigenvalues $1/2$ and $3/2$ respectively. Then, from the facts that the eigenspace of the eigenvalue $1/2$ is $1$-dimensional and the eigenspace of the eigenvalue $3/2$ is $8$-dimensional, the result follows from Theorem \ref{indjump}.
\end{proof}
Now, the two known deformations of the FNFN instanton on $\mathbb{R}^8$ are the translation and the dilation. It is clear that translation being $8$-dimensional, should come from spin representation, whereas dilation being one dimensional, should come from the trivial representation. From the fact that the eigenvalues of the twisted Dirac operator in the range $[1/2, 5/2]$ are $1/2$ and $3/2$, corresponding to the trivial and spin representations respectively, we should expect that the rate of dilation should be $1/2 - 5/2 = -2$ and that of translation should be $3/2-5/2 = -1$. This can be easily verified from the fact that the two deformations translation and dilation are given by $\iota_{\frac{\partial}{\partial x^i}}F_A$ and $\iota_{x^i\frac{\partial}{\partial x^i}}F_A$ respectively.

\bibliographystyle{plain}
\nocite {*}
\bibliography{references}
\addcontentsline{toc}{section}{References}

\end{document}